\documentclass[11pt, a4paper]{article}

\usepackage[english]{babel}
\usepackage[a4paper]{geometry}
\usepackage{amsmath,amssymb}
\usepackage{amsfonts}
\usepackage{amsthm}
\usepackage{graphicx}
\usepackage[utf8]{inputenc}
\usepackage{mathtools}
\usepackage{paralist}
\usepackage{color}
\usepackage{tikz}
\usetikzlibrary{matrix}

\numberwithin{equation}{section}

\theoremstyle{plain}
\newtheorem{theorem}{Theorem}[section]
\newtheorem{proposition}[theorem]{Proposition}
\newtheorem{lemma}[theorem]{Lemma}
\newtheorem{corollary}[theorem]{Corollary}

\theoremstyle{definition}

\newtheorem{remark}[theorem]{Remark}
\newtheorem{example}[theorem]{Example}

\newcommand{\C}{\mathbb{C}}
\newcommand{\R}{\mathbb{R}}

\newcommand{\bD}{\mathbb{D}}

\newcommand{\cE}{\mathcal{E}}

\newcommand{\cK}{\mathcal{K}}

\newcommand{\cO}{\mathcal{O}}

\newcommand{\coloneq}{\mathrel{\mathop:}=}

\newcommand{\conj}[1]{\overline{#1}}
\newcommand{\comp}[1]{#1^c}

\DeclarePairedDelimiter{\abs}{\lvert}{\rvert}

\DeclarePairedDelimiter{\cc}{[}{]}

\DeclarePairedDelimiter{\co}{[}{[}
\DeclarePairedDelimiter{\oo}{]}{[}

\DeclareMathOperator{\im}{Im}

\DeclareMathOperator{\exterior}{ext}

\DeclareMathOperator{\capacity}{cap}
\DeclareMathOperator{\sgn}{sgn}
\DeclareMathOperator{\trace}{trace}
\DeclareMathOperator{\wind}{wind}

\DeclareMathOperator{\Ri}{\mathcal{R}}

\title{Walsh's conformal map onto lemniscatic domains for polynomial 
pre-images I}
\author{Klaus Schiefermayr\footnotemark[1] \and Olivier 
S\`{e}te\footnotemark[2]}
\date{August 3, 2022}

\begin{document}
\maketitle

\renewcommand{\thefootnote}{\fnsymbol{footnote}}

\footnotetext[1]{University of Applied Sciences Upper Austria, Campus Wels, 
Austria, \\ \texttt{klaus.schiefermayr@fh-wels.at}}

\footnotetext[2]{Institute of Mathematics and Computer Science, Universit\"at 
Greifswald, Walther-Rathenau-Stra\ss{}e~47, 17489 Greifswald, Germany.
\texttt{olivier.sete@uni-greifswald.de}.
ORCID: 0000-0003-3107-3053}

\renewcommand{\thefootnote}{\arabic{footnote}}

\begin{abstract}
We consider Walsh's conformal map from the exterior of a compact set $E 
\subseteq \C$ onto a lemniscatic domain.  If $E$ is simply connected, the 
lemniscatic domain is the exterior of a circle, while if $E$ has several 
components, the lemniscatic domain is the exterior of a generalized lemniscate 
and is determined by the logarithmic capacity of $E$ and by the \emph{exponents} 
and \emph{centers} of the generalized lemniscate.
For general $E$, we characterize the exponents in terms of the Green's function 
of $\comp{E}$.  Under additional symmetry conditions on $E$, we also locate the 
centers of the lemniscatic domain.
For polynomial pre-images $E = P^{-1}(\Omega)$ of a simply-connected infinite 
compact set $\Omega$, we explicitly determine the exponents in the lemniscatic 
domain and derive a set of equations to determine the centers of the 
lemniscatic domain.
Finally, we present several examples where we explicitly obtain the exponents 
and centers of the lemniscatic domain, as well as the conformal map.
\end{abstract}



\section{Introduction}

The famous Riemann mapping theorem says that for any simply connected, compact 
and infinite set $E$ there exists a conformal map $\Ri_E : \comp{E} \coloneq 
\widehat{\C} \setminus E \to \comp{\overline{\bD}}$, where 
$\widehat{\C} = \C \cup \{ \infty \}$ denotes the extended complex plane,
$\bD$ the open and $\overline{\bD}$ the closed unit disk.
By imposing the normalization $\Ri_E(z) = \frac{z}{\capacity(E)} + \cO(1)$ as 
$z \to \infty$, where $\capacity(E)$ denotes the \emph{logarithmic capacity} of 
$E$, this map is unique.
In his 1956 article~\cite{Walsh1956}, J.\,L.~Walsh found the following 
canonical generalization for multiply connected domains.

\begin{theorem} \label{thm:walsh_map}
Let $E_1, \ldots, E_\ell \subseteq \C$ be disjoint simply connected, infinite 
compact sets and let
\begin{equation}
E = \bigcup_{j=1}^\ell E_j.
\end{equation}
In particular, $\comp{E} = \widehat{\C} \setminus E$ is an $\ell$-connected 
domain.
Then there exists a unique compact set of the form
\begin{equation} \label{eqn:lemniscatic_domain}
L \coloneq \{ w \in \C : \abs{U(w)} \leq \capacity(E) \}, \quad
U(w) \coloneq \prod_{j=1}^\ell (w-a_j)^{m_j},
\end{equation}
where $a_1, \ldots, a_\ell \in \C$ are distinct and $m_1, \ldots, m_\ell > 0$ 
are real numbers with $\sum_{j=1}^\ell m_j = 1$,
and a unique conformal map
\begin{equation} \label{eqn:Phi}
\Phi : \comp{E} \to \comp{L}
\end{equation}
normalized by
\begin{equation} \label{eqn:Phi_at_infinity}
\Phi(z) = z + \cO(1/z) \quad \text{at } \infty.
\end{equation}
If $E$ is bounded by Jordan curves, then $\Phi$ extends to a homeomorphism 
from $\overline{\comp{E}}$ to $\overline{\comp{L}}$.
\end{theorem}

\begin{remark}
\begin{enumerate}
\item 
By assumption, each $E_j$ satisfies $\capacity(E_j) > 0$ hence 
$\capacity(E) > 0$.

\item The points $a_1, \ldots, a_\ell$ (sometimes called `centers' of $L$) and 
also $m_1, \ldots, m_\ell$ in Theorem~\ref{thm:walsh_map} are uniquely 
determined.
The function $U$ is analytic in $\C \setminus \{ a_1, \ldots, a_\ell \}$ and
in general not single-valued, but its absolute value is single-valued.
Note that the compact set $L$, defined in~\eqref{eqn:lemniscatic_domain}, 
consists of $\ell$ disjoint compact components $L_1, \ldots, L_\ell$,
with $a_j \in L_j$ for $j = 1, \ldots, \ell$.
The components $L_1, \ldots, L_\ell$ are labeled such that a Jordan curve 
surrounding $E_j$ is mapped by $\Phi$ onto a Jordan curve surrounding $L_j$.

\item If $E$ is simply connected then the exterior Riemann map $\Ri_E : 
\comp{E} \to \comp{\overline{\bD}}$ with 
$\Ri_E(z) = d_1 z + d_0 + \cO(1/z)$ at $\infty$ and $d_1 = \Ri_E'(\infty) > 0$
and the Walsh map $\Phi$ are related by $\Ri_E(z) = d_1 \Phi(z) + d_0$, 
which follows from~\cite[Thm.~4]{Walsh1956}.
The corresponding lemniscatic domain is the disk $L = \{ w \in \C : \abs{w - 
a_1} \leq \capacity(E) \}$, where $a_1 = - d_0/d_1$ and $\capacity(E) = 1/d_1$.
This shows that Walsh's map onto lemniscatic domains is a canonical 
generalization of the Riemann map from simply to multiply connected domains.

\item The existence in Theorem~\ref{thm:walsh_map} was first shown 
by Walsh; see~\cite[Thm.~3]{Walsh1956} and the discussion below.
Other existence proofs were given by Grunsky~\cite{Grunsky1957a}, 
\cite{Grunsky1957b}, and~\cite[Thm.~3.8.3]{Grunsky1978}, 
and also by Jenkins~\cite{Jenkins1958} and Landau~\cite{Landau1961}.
However, these articles do not contain any analytic or numerical examples.
The first analytic examples were constructed by S\`ete and Liesen 
in~\cite{SeteLiesen2016}, and, subsequently,
a numerical method for computing the Walsh map was derived 
in~\cite{NasserLiesenSete2016} for sets bounded by smooth Jordan curves.

\item The domain $\comp{L}$ is usually called a \emph{lemniscatic domain}.  
This term seems to originate in Grunsky~\cite[p.~106]{Grunsky1978}.
\end{enumerate}
\end{remark}

In this paper, we bring some light into the computation of the parameters $m_j$ 
and $a_j$ appearing in Theorem~\ref{thm:walsh_map}.

In Section~\ref{sect:general}, as a first main result, we derive a general 
formula (Theorem~\ref{thm:mj_by_integration}) for the exponents $m_j$ in terms 
of the Green's function of $\comp{E}$, denoted by $g_E$.
Of special interest is of course the case when $E$ is real or when $E$ or some 
component $E_j$ are symmetric with respect to the real line, i.e., $E^* = E$ or 
$E_j^* = E_j$, where
\begin{equation} \label{eqn:complex_conjugate_set}
K^* \coloneq \{ z \in \widehat{\C} : \conj{z} \in K \}
\end{equation}
denotes the complex conjugate of a set $K \subseteq \widehat{\C}$.
We prove that $E^* = E$ and $E_j^* = E_j$ implies that $a_j \in \R$ 
(Theorem~\ref{thm:aj_real}).  In the 
case that all components are symmetric, we give an interlacing property of the 
components $E_j$ and the critical points of $g_E$ 
(Theorem~\ref{thm:crit_pts_of_gE}).

In Section~\ref{sect:poly_preimages}, we consider the case when $E$ is a 
polynomial pre-image of a simply connected compact infinite set $\Omega$, that 
is, $E = P_n^{-1}(\Omega)$.  In this case, we prove in 
Theorem~\ref{thm:exponents_poly_preimage} that the $m_j$ are always 
rational of the form $m_j = n_j/n$, where $n$ is the degree of the polynomial 
$P_n$ and $n_j$ is the number of zeros of $P_n(z) - \omega$ in $E_j$, where 
$\omega \in \Omega$.
Moreover, the unknowns $a_1, \ldots, a_\ell$ are characterized by a system of 
equations which in particular can be solved explicitly in the case $\ell = 2$.
With the help of these findings, we obtain an analytic expression for the map 
$\Phi$ if $P_n^{-1}(\Omega)$ is connected 
(Corrolary~\ref{cor:connected_pre-image}).

Finally, Section~\ref{sect:examples} contains several illustrative 
examples when $E = P_n^{-1}(\Omega)$ and when $\Omega = \overline{\bD}$, 
$\Omega = \cc{-1, 1}$ or when $\Omega$ is a Chebyshev ellipse.
In particular, we determine the exponents and centers of the 
corresponding lemniscatic domain and visualize the conformal map $\Phi$.

\section{Results for general compact sets}
\label{sect:general}

Let the notation be as in Theorem~\ref{thm:walsh_map}.
The Green's function (with pole at $\infty$) of $\comp{L}$ is
\begin{equation} \label{eqn:g_cL}
g_L(w) = \log \abs{U(w)} - \log(\capacity(E))
= \sum_{j=1}^\ell m_j \log \abs{w-a_j} - \log(\capacity(E))
\end{equation}
since $g_L$ is harmonic in $\C \setminus \{ a_1, \ldots, a_\ell \}$, is 
zero on $\partial (\comp{L})$, and $g_L(w) - \log \abs{w}$ is harmonic at 
$\infty$ 
with $\lim_{w \to \infty} (g_L(w) - \log \abs{w}) = -\log(\capacity(E))$.
Then the Green's function of $\comp{E}$ is
\begin{equation} \label{eqn:g_cK}
g_E(z) = g_L(\Phi(z)), \quad z \in \comp{E},
\end{equation}
since $\Phi : \comp{E} \to \comp{L}$ is 
conformal with $\Phi(z) = z + \cO(1/z)$ at 
$\infty$.
In particular, $\capacity(E) = \capacity(L)$.
Denote for $R > 1$ the level curves of $g_E$ and $g_L$ by
\begin{equation*}
\Gamma_R = \{ z \in \comp{E} : g(z) = \log(R) \},
\quad
\Lambda_R = \{ w \in \comp{L} : g_L(w) = \log(R) \}.
\end{equation*}
Then $\Phi(\Gamma_R) = \Lambda_R$ and $\Phi$ maps the exterior of 
$\Gamma_R$ onto the exterior of $\Lambda_R$.
Let $R_* > 1$ be the largest number, such that $g_E$ has no critical point 
interior to $\Gamma_{R_*}$ (if $\ell = 1$, then $R_* = \infty$; see 
Theorem~\ref{thm:Green_function_connectivity} below).  Then $\Phi$ is the 
conformal map of $\exterior(\Gamma_R)$ onto the lemniscatic domain 
$\exterior(\Lambda_R)$ for all $1 < R < R_*$; see also~\cite[p.~31]{Walsh1958}.

Here and in the following, we extensively use the \emph{Wirtinger derivatives}
\begin{equation*}
\partial_z = \frac{1}{2} (\partial_x - i \partial_y) \quad \text{and} \quad
\partial_{\conj{z}} = \frac{1}{2} (\partial_x + i \partial_y),
\end{equation*}
where $z = x + iy$ with $x, y \in \R$.
We relate the exponents and centers of the lemniscatic domain to the Wirtinger 
derivatives $\partial_z g_E$ and $\partial_w g_L$ of the Green's functions.
Note that $\partial_z g$ is analytic if $g$ is a harmonic function, since
then $\partial_{\conj{z}} (\partial_z g) = \frac{1}{4} \Delta g = 0$.

\begin{lemma} \label{lem:int_transform}
The Green's functions $g_L$ and $g_E$ from~\eqref{eqn:g_cL} 
and~\eqref{eqn:g_cK} satisfy
\begin{equation} \label{eqn:chain_rule}
\partial_z g_E(z) = \partial_w g_L(\Phi(z)) \cdot \Phi'(z).
\end{equation}
Moreover, if $\gamma : \cc{a, b} \to \comp{E}$ is a smooth path, then
\begin{equation} \label{eqn:int_transform}
\int_\gamma \partial_z g_E(z) \, dz = \int_{\Phi \circ \gamma} \partial_w 
g_L(w) \, dw.
\end{equation}
\end{lemma}

\begin{proof}
Since $\Phi$ is analytic, we have $\partial_z \Phi = \Phi'$ and 
$\partial_{\conj{z}} \Phi = 0$.
Moreover, $\partial_z \conj{\Phi} = \conj{\partial_{\conj{z}} \Phi} = 0$.
With the chain rule for the Wirtinger derivatives and~\eqref{eqn:g_cK}, we find
\begin{equation} 
\frac{\partial g_E}{\partial z}(z)
= \frac{\partial g_L}{\partial w}(\Phi(z)) \cdot \frac{\partial 
\Phi}{\partial z}(z)
+ \frac{\partial g_L}{\partial \conj{w}}(\Phi(z)) \cdot \frac{\partial 
\conj{\Phi}}{\partial z}(z)
= \partial_w g_L(\Phi(z)) \cdot \Phi'(z),
\end{equation}
which is~\eqref{eqn:chain_rule}.
Integrating this expression over $\gamma$ yields
\begin{equation}
\int_\gamma \partial_w g_L(\Phi(z)) \Phi'(z) \, dz
= \int_a^b \partial_w g_L(\Phi(\gamma(t))) \Phi'(\gamma(t)) \gamma'(t) \, 
dt
= \int_{\Phi \circ \gamma} \partial_w g_L(w) \, dw.
\end{equation}
In combination with~\eqref{eqn:chain_rule}, this 
yields~\eqref{eqn:int_transform}.
\end{proof}

\begin{remark}
Formally, equation~\eqref{eqn:chain_rule} yields the relation between 
differentials
\begin{equation}
\partial_z g_E(z) \, dz
= \partial_w g_L(\Phi(z)) \cdot \Phi'(z) \, dz
= \partial_w g_L(w) \, dw,
\end{equation}
which yields~\eqref{eqn:int_transform} upon integrating.
\end{remark}

We are now ready to express the exponents $m_j$ through the Wirtinger 
derivatives of the Green's function.
For $j = 1, \ldots, \ell$, let $\gamma_j$ be a closed curve in $\C \setminus 
E$ with $\wind(\gamma_j; z) = \delta_{jk}$ for $z \in E_k$ and $k = 1, \ldots, 
\ell$, where $\wind(\gamma; z_0)$ denotes the winding number of the 
curve $\gamma$ about $z_0$, and $\delta_{jk}$ is the usual Kronecker delta.
More informally, the curve $\gamma_j$ contains $E_j$ but no 
$E_k$, $k \neq j$, in its interior.

\begin{theorem} \label{thm:mj_by_integration}
In the notation of Theorem~\ref{thm:walsh_map},
let $g_E$ and $g_L$ be the Green's functions of $\comp{E}$ 
and $\comp{L}$, respectively.
For each $j \in \{ 1, \ldots, \ell \}$, let $\gamma_j$ be a
closed curve in $\C \setminus E$ with $\wind(\gamma_j; z) = \delta_{jk}$ for $z 
\in E_k$ and $k = 1, \ldots, \ell$, and let $\lambda_j = \Phi \circ \gamma_j$.
Then,
\begin{equation} \label{eqn:mj_by_integration}
m_j
= \frac{1}{2 \pi i} \int_{\lambda_j} 2 \partial_w g_L(w) \, dw
= \frac{1}{2 \pi i} \int_{\gamma_j} 2 \partial_z g_E(z) \, dz.
\end{equation}
Moreover, if the function $f$ is analytic interior to $\lambda_j$ and 
continuous on $\trace(\lambda_j)$, then
\begin{equation} \label{eqn:aj_by_integration}
m_j f(a_j)
= \frac{1}{2 \pi i} \int_{\lambda_j} f(w) 2 \partial_w g_L(w) \, dw
= \frac{1}{2 \pi i} \int_{\gamma_j} f(\Phi(z)) 2 \partial_z g_E(z) \, dz.
\end{equation}
\end{theorem}

\begin{proof}
Since $2 \partial_w \log \abs{w} = \partial_w \log(w \conj{w}) = 1/w$, we
obtain from~\eqref{eqn:g_cL} that
\begin{equation} \label{eqn:deriv_gL}
2 \partial_w g_L(w) = \sum_{j=1}^\ell \frac{m_j}{w - a_j},
\end{equation}
which is a rational function.
By construction, $\lambda_j$ is a closed curve in $\C \setminus L$ with 
$\wind(\lambda_j; a_k) = \delta_{jk}$.
Integrating over $\lambda_j$ yields the first equality 
in~\eqref{eqn:mj_by_integration}.
The second equality follows by Lemma~\ref{lem:int_transform}.
Using~\eqref{eqn:deriv_gL} and the residue theorem, we obtain
\begin{equation}
\frac{1}{2 \pi i} \int_{\lambda_j} f(w) 2 \partial_w g_L(w) \, dw
= \frac{1}{2 \pi i} \int_{\lambda_j} \sum_{s=1}^\ell \frac{m_s f(w)}{w - 
a_s} \, dw
= m_j f(a_j).
\end{equation}
This shows the first equality in~\eqref{eqn:aj_by_integration}.
Multiplying~\eqref{eqn:chain_rule} by $f(\Phi(z))$ and integrating yields the 
second equality in~\eqref{eqn:aj_by_integration}.
\end{proof}

\begin{remark}
\begin{enumerate}
\item By~\eqref{eqn:mj_by_integration} in Theorem~\ref{thm:mj_by_integration}, 
the exponent $m_j$ of the lemniscatic domain is the residue of $2 \partial_w 
g_L$ at $a_j$.
Moreover, $m_j$ is (up to the factor $\frac{1}{2 \pi i}$) the \emph{module of 
periodicity} (or \emph{period}) of the differential $2 \partial_z g_E(z) \, 
dz$; see~\cite[p.~147]{Ahlfors1979}.
The latter can be rewritten as
\begin{equation*}
\int_{\gamma_j} 2 \partial_z g_E(z) \, dz
= \int_{\gamma_j} \left( - \frac{\partial g_E}{\partial y} dx + 
\frac{\partial g_E}{\partial x} dy \right)
= \int_{\gamma_j} \frac{\partial g_E}{\partial n}(z) \, \abs{dz},
\end{equation*}
where the middle integral is over the \emph{conjugate differential} of $d g_E$, 
and where $\frac{\partial g_E}{\partial n}$ is the derivative with respect to 
the normal pointing to the right of $\gamma_j$; 
see~\cite[pp.~162--164]{Ahlfors1979} for a detailed discussion.

\item 
Since $\partial_z g_E$ is analytic in $\comp{E}$ and 
$\partial_w g_L$ is analytic in $\widehat{\C} \setminus \{ a_1, \ldots, a_\ell 
\} \supseteq \comp{L}$, the integrals 
in~\eqref{eqn:mj_by_integration} have the same value for all positively 
oriented 
closed curves that contain only $E_j$ or $a_j$ in their interior.
\end{enumerate}
\end{remark}

The following well-known result due to J.\,L.\,Walsh~\cite{Walsh1969} 
establishes a relation between the critical points of the Green's function and 
the connectivity of $\comp{E}$.

\begin{theorem}[{\cite[pp.~67--68]{Walsh1969}}] 
\label{thm:Green_function_connectivity}
Let $E \subseteq \C$ be compact such that $\cK = \comp{E}$ is 
connected and such that $\cK$ possesses a Green's function $g_E$ with pole at 
infinity.
If $\cK$ is of finite connectivity $\ell$, then $g_E$ has precisely $\ell-1$ 
critical points in $\C \setminus E$, counted according to their multiplicity.
If $\cK$ is of infinite connectivity, $g_E$ has a countably infinite number 
of critical points.
Moreover, all critical points of $g_E$ lie in the convex hull of $E$.
\end{theorem}

As is typical for conformal maps with $\Phi(z) = z + \cO(1/z)$ at $\infty$, 
symmetry of $E$ (e.g., rotational symmetry or symmetry with respect to the real 
line) leads to the same symmetry of $L$, and to ``symmetry'' in the map 
$\Phi$.

\begin{lemma}[{\cite[Lem.~2.2]{SeteLiesen2016}}] \label{lem:symmetry_E_Phi}
Let the notation be as in Theorem~\ref{thm:walsh_map}.  Then the following 
symmetry relations hold.
\begin{enumerate}
\item If $E = E^*$, then $L = L^*$ and 
$\Phi(z) = \conj{\Phi(\conj{z})}$.
\item If $E = e^{i \theta} E \coloneq \{ e^{i \theta} z : z \in E \}$, then 
$L = e^{i \theta} L$ and $\Phi(z) = e^{-i \theta} \Phi(e^{i \theta} z)$.
\item In particular:
If $E = -E = \{ -z : z \in E \}$, then $L = - L$ and $\Phi(z) = -\Phi(-z)$.
\end{enumerate}
\end{lemma}

In the last two results of this section, we consider the case where $E$ and one 
or all of its components $E_j$ are symmetric with respect to the real line.
This allows to locate the points $a_1, \ldots, a_\ell$ and the critical points 
of the Green's function $g_E$.

\begin{theorem} \label{thm:aj_real}
In the notation of Theorem~\ref{thm:walsh_map}, suppose that $E^* = E$.
Let $j \in \{ 1, \ldots, \ell \}$.  If $E_j^* = E_j$ then $a_j \in \R$.
\end{theorem}

\begin{proof}
Since $E^* = E$, we have $\Phi(\conj{z}) = \conj{\Phi(z)}$ 
by Lemma~\ref{lem:symmetry_E_Phi} and $\partial_z g_E(\conj{z}) = 
\conj{\partial_z g_E(z)}$ by Lemma~\ref{lem:symmetry_of_Green_function}.
Next, if $E_j^* = E_j$ for some $j \in \{ 1, \ldots, \ell \}$ then there 
exists a smooth Jordan curve $\gamma_j$ in $\C \setminus E$ symmetric with 
respect to the real line which surrounds $E_j$ in the positive sense, but no 
other component $E_k$, $k \neq j$, i.e., $\wind(\gamma_j; z) = \delta_{jk}$ for 
$z \in E_k$ and $k = 1, \ldots, \ell$.  By~\eqref{eqn:aj_by_integration},
\begin{equation*}
m_j a_j = \frac{1}{2 \pi i} \int_{\gamma_j} \Phi(z) 2 \partial_z g_E(z) \, dz,
\end{equation*}
where $\Phi(\conj{z}) 2 \partial_z g_E(\conj{z}) = \conj{\Phi(z) 2 
\partial_z 
g_E(z)}$ on $\gamma_j$.  By Lemma~\ref{lem:symmetric_integral_is_real},
we obtain $m_j a_j \in \R$, hence $a_j \in \R$ since $m_j > 0$.
\end{proof}

In Theorem~\ref{thm:aj_real}, if a component $E_j$ is not symmetric with 
respect to the real line, then the corresponding point $a_j$ is in general not 
real, as the example of the star in~\cite[Cor.~3.3]{SeteLiesen2016} shows.

% QUESTION (or to be deleted):
% We use $E^* = E$ in our proof, but is it necessary?  In other words, if $E = 
% E_1 \cup E_2$ with $E_1^* = E_1$ but $E^* \neq E$, i.e., $E_2^* \neq E_2$, 
% can it happen that $a_1 \notin \R$?

% We note for further reference that the integration by parts formula carries 
% over to path integrals upon noting that
% \begin{equation}
% \int_\gamma f'(z) g(z) \, dz
% % = \int_a^b f'(\gamma(t)) g(\gamma(t)) \gamma'(t) \, dt
% = \int_a^b (f \circ \gamma)'(t) (g \circ \gamma)(t) \, dt.
% \end{equation}
% 
% \begin{lemma} \label{lem:integration_by_parts}
% Let $\gamma : \cc{a, b} \to \C$ be continuously differentiable and $f, g$ be 
% analytic on the trace of $\gamma$.  Then
% \begin{equation}
% \int_\gamma f'(z) g(z) \, dz
% = f(z) g(z) \big\vert_{\gamma(a)}^{\gamma(b)} - \int_\gamma f(z) g'(z) \, dz.
% \end{equation}
% The boundary term vanishes when $\gamma$ is closed.
% \end{lemma}

If $E_j^* = E_j$ for all components of $E$ then we order the components
``from left to right'':  By Lemma~\ref{lem:compact_cap_R_is_connected}, each 
$E_j \cap \R$ is a point or an interval, and we label $E_1, \ldots, E_\ell$
such that $x \in E_j \cap \R$ and $y \in E_{j+1} \cap \R$ implies $x < y$ for 
all $j = 1, \ldots, \ell-1$.

\begin{theorem} \label{thm:crit_pts_of_gE}
Let $E = E_1 \cup \ldots \cup E_\ell$ be as in Theorem~\ref{thm:walsh_map} and 
suppose that $E_j^* = E_j$ for all $j = 1, \ldots, \ell$.
Then the following hold.
\begin{enumerate}
\item The $\ell-1$ critical points of the Green's function $g_E$ are real.
Moreover, each $E_j$ intersects $\R$ in a point or an interval, and the 
critical points of $g_E$ interlace the sets $E_j \cap \R$, $j = 1, \ldots, 
\ell$.

\item If $E_1, \ldots, E_\ell$ are ordered ``from left to right'' then
$a_1 < a_2 < \ldots < a_\ell$.
\end{enumerate}
\end{theorem}

\begin{proof}
(i)
For each $j = 1, \ldots, \ell$, the set $E_j \cap \R$ is a point or an 
interval by Lemma~\ref{lem:compact_cap_R_is_connected}.
For $j = 1, \ldots, \ell - 1$, denote the `gap' on the real line between 
$E_j$ and $E_{j+1}$ by
\begin{equation*}
I_j \coloneq \oo{ \max (E_j \cap \R), \min (E_{j+1} \cap \R)}, \quad j = 1, 
\ldots, \ell-1.
\end{equation*}
The Green's function $g_E$ is positive on $I_j$ and can be continuously 
extended to $\conj{I}_j$ with boundary values $0$.
Then $g_E$ has a maximum on $\conj{I}_j$ at a point $x_j \in I_j$ at which
$\partial_x g_E(x_j) = 0$.
By~\eqref{eqn:partial_deriv_of_green_function}, we have 
$\partial_y g_E(x_j) = - \partial_y g_E(x_j)$, i.e., $\partial_y g_E(x_j) = 0$.
This shows that $x_j$ is a critical point of $g_E$ for $j = 1, \ldots, 
\ell-1$.
These are the $\ell-1$ critical points of $g_E$ which are real and interlace 
the sets $E_j \cap \R$.

(ii)
Since $E = E^*$, we have $\Phi(z) = \conj{\Phi(\conj{z})}$ by 
Lemma~\ref{lem:symmetry_E_Phi}.  In particular, $\Phi$ maps $\R \setminus E$ 
onto $\R \setminus L$.  Since $\Phi(z) = z + \cO(1/z)$ at infinity, $\Phi$ maps 
$I_0 \coloneq \oo{- \infty, \min(E_1 \cap \R)}$ onto $J_0 \coloneq \oo{-\infty, 
\min(L \cap \R)}$.
Let $\gamma_1$ be a Jordan curve in $\C \setminus E$ which surrounds $E_1$ in 
the positive sense, but no other component $E_k$, $k \neq 1$.
Then $\gamma_1$ intersects $I_0$ and $I_1$ (see~(i)), hence the curve 
$\Phi(\gamma_1)$ intersects the images $J_0 = \Phi(I_0)$ and $J_1 \coloneq 
\Phi(I_1)$.  This shows that $L_1$ is the leftmost component of $L$ and $a_1$ 
is the minimum of $a_1, \ldots, a_\ell$. 
Proceeding in a similar way gives that the components $L_1, \ldots, L_\ell$ 
are ordered from left to right, and therefore $a_1 < a_2 < \ldots < a_\ell$.
\end{proof}

\section{Results for polynomial pre-images}
\label{sect:poly_preimages}

Let $\Omega \subseteq \C$ be a compact infinite set such 
that $\comp{\Omega}$ is a simply connected domain in 
$\widehat{\C}$ and let $\Ri_\Omega$ be the exterior Riemann map of 
$\Omega$, i.e., the conformal map
\begin{equation} \label{eqn:riemannmap}
\Ri_\Omega : \comp{\Omega} \to \comp{\overline{\bD}}
\quad \text{with }
\Ri_\Omega(z) = d_1 z + d_0 + \sum_{k=1}^\infty \frac{d_{-k}}{z^k}
\quad \text{for } \abs{z} > R,
\end{equation}
where
\begin{equation} \label{eqn:d1}
d_1 = \Ri_\Omega'(\infty) = \frac{1}{\capacity(\Omega)} > 0,
\end{equation}
$R \coloneq \max_{z \in \Omega} \abs{z}$,
and $\bD = \{ z \in \C : \abs{z} < 1 \}$ is the open unit disk.
By~\cite[Thm.~4.4.4]{Ransford1995}, the Green's function of $\comp{\Omega}$ is
\begin{equation}
g_\Omega(z) = \log \abs{\Ri_\Omega(z)}.
\end{equation}
Let $P_n$ be a polynomial of degree $n \geq 1$, more precisely,
\begin{equation} \label{eqn:Pn}
P_n(z) = \sum_{j=0}^n p_j z^j \quad \text{with } p_n \in \C \setminus \{ 0 \},
\end{equation}
and consider the pre-image of $\Omega$ under $P_n$, that is
\begin{equation} \label{eqn:E_poly_preimage}
E = P_n^{-1}(\Omega) = \{ z \in \C : P_n(z) \in \Omega \}.
\end{equation}
The set $E$ is compact and, by Theorem~\ref{thm:cK_is_connected}, the 
complement $\comp{E}$ is connected.
Therefore, the Green's function of $\comp{E}$ is
\begin{equation} \label{eqn:green_function_poly_preimage}
g_E(z) = \frac{1}{n} g_\Omega(P_n(z))
= \frac{1}{n} \log \abs{\Ri_\Omega(P_n(z))},
\end{equation}
see~\cite[p.~134]{Ransford1995}.
Since $2 \partial_z \log \abs{f} = f'/f$ for an analytic function $f$, we have
\begin{equation} \label{eqn:dzbar_green_function_poly_preimage}
2 \partial_z g_E(z) = \frac{1}{n} \frac{\Ri_\Omega'(P_n(z)) 
P_n'(z)}{\Ri_\Omega(P_n(z))}.
\end{equation}
The logarithmic capacity of $E$ is
\begin{equation} \label{eqn:capacity_poly_preimage}
\capacity(E) = \capacity(P_n^{-1}(\Omega))
= \left( \frac{\capacity(\Omega)}{\abs{p_n}} \right)^{1/n}
% = \left( \frac{1}{d_1 \abs{p_n}} \right)^{1/n}
= \frac{1}{\sqrt[n]{d_1 \abs{p_n}}},
\end{equation}
see~\cite[Thm.~5.2.5]{Ransford1995}.

By Theorem~\ref{thm:Green_function_connectivity}, the number of components of 
$E$ can be characterized as follows.
For the case $\Omega = \cc{-1, 1}$, see also~\cite[Thm.~4 and 
Thm.~5]{Schiefermayr2012}.

\begin{theorem} \label{thm:E_decomposition}
The set $E$ in~\eqref{eqn:E_poly_preimage} consists of $\ell$ disjoint simply 
connected compact components $E_1, \ldots, E_\ell$, i.e.,
\begin{equation} \label{eqn:E_components}
E = P_n^{-1}(\Omega) = \bigcup_{j=1}^\ell E_j,
\end{equation}
if and only if $P_n$ has exactly $\ell-1$ critical points $z_1, \ldots, 
z_{\ell-1}$ (counting multiplicities) for which $P_n(z_k) \notin \Omega$ for $k 
= 1, \ldots, \ell-1$.
Moreover, the number of zeros of $P_n(z) - \omega$ in $E_j$ is the same for all 
$\omega \in \Omega$, and this number is denoted by $n_j$.
\end{theorem}

\begin{proof}
By Theorem~\ref{thm:Green_function_connectivity}, $E$ has $\ell$ components if 
and only if $g_E$ has $\ell - 1$ critical points (in $\C \setminus E$).
Since $g_E$ is real-valued, $z_0 \in \C \setminus E$ is a critical point of 
$g_E$ if and only if $\partial_z g_E(z_0) = 0$.
By~\eqref{eqn:dzbar_green_function_poly_preimage}, the latter is 
equivalent to $P_n'(z_0) = 0$.

For $j = 1, \ldots, \ell$, let $\gamma_j$ be a Jordan curve in $\C \setminus E$ 
with $\wind(\gamma_j; z) = \delta_{jk}$ for $z \in E_k$ and $k = 1, \ldots, 
\ell$.
Let $z_0 \in E_j$ and $\omega_0 \coloneq P_n(z_0) \in \Omega$, then,
by the argument principle,
\begin{equation}
n_j \coloneq \wind(P_n \circ \gamma_j; \omega_0)
= \abs{ \{ z \in E_j : P_n(z) = \omega_0 \}} \geq 1.
\end{equation}
Since $P_n \circ \gamma_j$ is a closed curve in $\C \setminus \Omega$, we 
have $\wind(P_n \circ \gamma_j; \omega) = \wind(P_n \circ \gamma_j; \omega_0)$ 
for all $\omega \in \Omega$, i.e., every point in $\Omega$ has exactly $n_j$ 
pre-images under $P_n$ in $E_j$.
\end{proof}

In the rest of this section, we assume that $E$ has $\ell$ components $E_1, 
\ldots, E_\ell$, i.e., that $P_n$ has exactly $\ell-1$ critical points with 
critical values in $\C \setminus \Omega$.

\begin{theorem} \label{thm:exponents_poly_preimage}
Let $E = P_n^{-1}(\Omega)$ and the numbers $n_1, \ldots, n_\ell$ be defined as 
in Theorem~\ref{thm:E_decomposition}.
Then the exponents $m_j$ in the lemniscatic domain in 
Theorem~\ref{thm:walsh_map} are given by
\begin{equation}
m_j = \frac{n_j}{n}, \quad j = 1, \ldots, \ell.
\end{equation}
\end{theorem}

\begin{proof}
For $j = 1, \ldots, \ell$, let $\gamma_j$ be a positively oriented Jordan 
curve in $\C \setminus E$ with $\wind(\gamma_j; z) = \delta_{jk}$ for $z \in 
E_k$ and $k = 1, \ldots, \ell$.
Using~\eqref{eqn:mj_by_integration} 
and~\eqref{eqn:dzbar_green_function_poly_preimage}, we obtain
\begin{equation}
m_j = \frac{1}{2 \pi i} \int_{\gamma_j} 2 \partial_z g_E(z) \, dz
= \frac{1}{n} \frac{1}{2 \pi i} \int_{\gamma_j} 
\frac{\Ri_\Omega'(P_n(z)) P_n'(z)}{\Ri_\Omega(P_n(z))} \, dz.
\end{equation}
Substituting $u = P_n(z)$ yields
\begin{equation} \label{eqn:mj_poly_preim_substituted}
m_j = \frac{1}{n} \frac{1}{2 \pi i} \int_{P_n \circ \gamma_j} 
\frac{\Ri_\Omega'(u)}{\Ri_\Omega(u)} \, du.
\end{equation}
Since $\wind(P_n \circ \gamma_j; u_0) = n_j$ for $u_0 \in \Omega$, the integral 
in~\eqref{eqn:mj_poly_preim_substituted} can be replaced by $n_j$ times an 
integral over a positively oriented Jordan curve $\Gamma$ in $\C \setminus 
\Omega$, i.e.,
\begin{equation}
m_j = \frac{n_j}{n} \frac{1}{2 \pi i} \int_{\Gamma} 
\frac{\Ri_\Omega'(u)}{\Ri_\Omega(u)} \, du.
\end{equation}
The integral is
\begin{equation}
\frac{1}{2 \pi i} \int_{\Gamma} \frac{\Ri_\Omega'(u)}{\Ri_\Omega(u)} \, du
= \frac{1}{2 \pi i} \int_{\Gamma} \frac{\frac{(u - u_0) 
\Ri_\Omega'(u)}{\Ri_\Omega(u)}}{u - u_0} \, du
= \lim_{u \to \infty} \frac{(u - u_0) \Ri_\Omega'(u)}{\Ri_\Omega(u)}
= 1
\end{equation}
by Cauchy's integral formula for an infinite domain;
see, e.g.,~\cite[Problem~14.14]{Markushevich1965}.
\end{proof}

% \subsection{Relating the Walsh and the Riemann map}

Next, we derive a relation between the Walsh map $\Phi$ and the Riemann map 
$\Ri_\Omega$.
Let $E = P_n^{-1}(\Omega)$ be as in~\eqref{eqn:E_poly_preimage}.
Liesen and the second author proved in~\cite[Eqn.~(3.2)]{SeteLiesen2016} that 
the lemniscatic map $\Phi$ in Theorem~\ref{thm:walsh_map} and 
the exterior Riemann map $\Ri_\Omega$ are related by
\begin{equation} \label{eqn:relation_between_maps_with_U}
\abs{U(\Phi(z))} = \capacity(E) \abs{\Ri_\Omega(P_n(z))}^{1/n}, \quad z \in 
\comp{E},
\end{equation}
with $U$ from~\eqref{eqn:lemniscatic_domain}.  This follows by considering the
identity~\eqref{eqn:g_cK} between the corresponding Green's functions.
In Theorem~\ref{thm:relation_between_maps}, we establish a stronger result.

By Theorem~\ref{thm:exponents_poly_preimage}, the exponents of $U$ satisfy
$m_j = n_j/n$.
Together with~\eqref{eqn:capacity_poly_preimage}, we see that
\begin{equation} \label{eqn:Q}
Q(w) \coloneq \frac{e^{i \arg(p_n)}}{\capacity(E)^n} U(w)^n
= d_1 p_n \prod_{j=1}^\ell (w-a_j)^{n_j}
\end{equation}
is a polynomial of degree $n$.  Note that $L = \{ w \in \C : \abs{Q(w)} \leq 1 
\}$, and $Q : \comp{L} \to \comp{\overline{\bD}}$ is an $n$-to-$1$ map.
Then, equation~\eqref{eqn:relation_between_maps_with_U} is equivalent to
\begin{equation} \label{eqn:modulus_of_maps_is_equal}
\abs{Q(\Phi(z))} = \abs{\Ri_\Omega(P_n(z))}, \quad z \in \comp{E}.
\end{equation}
Next, we show that equality is also valid without the absolute value.
Moreover, we derive a relation between the points $a_j$ and the coefficients 
$p_{n-1}$, $p_n$ of $P_n$ for $n \geq 2$.  The case $n = 1$ is discussed in 
Remark~\ref{rem:poly_preimage_for_deg_one}.

\begin{theorem} \label{thm:relation_between_maps} \label{thm:aj_at_infty}
Let $E = P_n^{-1}(\Omega)$ be as in~\eqref{eqn:E_poly_preimage}.  We then have
\begin{equation} \label{eqn:relation_between_maps}
Q(\Phi(z)) = \Ri_\Omega(P_n(z)), \quad z \in \comp{E},
\end{equation}
that is,
\begin{equation} \label{eqn:formal_expression_Phi}
\Phi = Q^{-1} \circ \Ri_\Omega \circ P_n,
\end{equation}
with that branch of $Q^{-1}$ such that $\Phi(z) = z + \cO(1/z)$ at $\infty$.
Moreover, for $n \geq 2$,
\begin{equation} \label{eqn:aj_at_infty}
\sum_{j=1}^\ell n_j a_j = - \frac{p_{n-1}}{p_n}.
\end{equation}
\end{theorem}

\begin{proof}
Consider the Laurent series at infinity of $\Ri_\Omega \circ P_n$ and $Q \circ 
\Phi$.  By~\eqref{eqn:riemannmap} and~\eqref{eqn:Pn},
\begin{equation} \label{eqn:Riemann_of_Pn}
\Ri_\Omega(P_n(z))
= d_1 P_n(z) + d_0 + \sum_{k=1}^\infty \frac{d_{-k}}{P_n(z)^k}
= d_1 p_n z^n + d_1 p_{n-1} z^{n-1} + \cO(z^{n-2}).
\end{equation}
Since $\Phi(z) = z + \cO(1/z)$ at infinity, we have
\begin{equation*}
(\Phi(z) - a_j)^{n_j}
% = (z-a_j + O(1/z))^{n_j}
= (z-a_j)^{n_j} + \cO(z^{n_j-2})
= z^{n_j} - n_j a_j z^{n_j-1} + \cO(z^{n_j-2})
\end{equation*}
and, by~\eqref{eqn:Q},
\begin{equation} \label{eqn:QPhi}
Q(\Phi(z))
= d_1 p_n \prod_{j=1}^\ell (\Phi(z) - a_j)^{n_j}
% &= d_1 p_n \prod_{j=1}^\ell (z^{n_j} - n_j a_j z^{n_j-1} + O(z^{n_j-2})) \\
= d_1 p_n z^n - d_1 p_n \sum_{j=1}^\ell n_j a_j z^{n-1} + \cO(z^{n-2}).
\end{equation}
The function $(Q \circ \Phi)/(\Ri_\Omega \circ P_n)$ is analytic in $\C 
\setminus E$ with constant modulus one, 
see~\eqref{eqn:modulus_of_maps_is_equal}, therefore constant (maximum modulus 
principle) and
\begin{equation}
Q(\Phi(z)) = c \Ri_\Omega(P_n(z)), \quad z \in \comp{E},
\end{equation}
where $c \in \C$ with $\abs{c} = 1$.
By comparing the coefficients of $z^n$ of the Laurent series at $\infty$, we 
see that $c = 1$, which shows~\eqref{eqn:relation_between_maps}.
Comparing the coefficients of $z^{n-1}$ then yields~\eqref{eqn:aj_at_infty}.
\end{proof}

\begin{figure}
{\centering
\begin{tikzpicture}
\matrix (m) [matrix of math nodes,row sep=3em,column sep=4em,minimum width=2em]
{ z \in \comp{E} & \comp{\Omega} \\
 w \in \comp{L} & \comp{\overline{\bD}} 
\\};
\path[-stealth]
    (m-1-1) edge node [left] {$\Phi$} (m-2-1)
            edge node [above] {$P_n$} (m-1-2)
    (m-2-1) edge node [below] {$Q$} (m-2-2)
    (m-1-2) edge node [right] {$\Ri_\Omega$} (m-2-2);
\end{tikzpicture}

}
\caption{Commutative diagram of the maps in 
Theorem~\ref{thm:relation_between_maps}.}
\label{fig:relation_between_maps}
\end{figure}

Figure~\ref{fig:relation_between_maps} illustrates 
Theorem~\ref{thm:relation_between_maps}.

\begin{remark} \label{rem:poly_preimage_for_deg_one}
In the case $n = 1$, i.e., $P_1(z) = p_1 z + p_0$ is a linear transformation,
the conformal map and lemniscatic domain are given explicitly as follows.
In this case, $E = P_1^{-1}(\Omega)$ consists of a single component, i.e., 
$\ell = 1$ and $m_1 = 1$, and $Q(w) = d_1 p_1 (w-a_1)$.
Comparing the constant terms at infinity of
$\Ri_\Omega(P_n(z)) = d_1 p_1 z + (d_1 p_0 + d_0) + \cO(1/z)$ 
with $Q(\Phi(z))$ from~\eqref{eqn:QPhi} yields
\begin{equation}
a_1 = - \frac{d_1 p_0 + d_0}{d_1 p_1}.
\end{equation}
By Theorem~\ref{thm:relation_between_maps}, the conformal map $\Phi : 
\comp{E} \to \comp{L}$ is
\begin{equation}
\Phi(z) = (Q^{-1} \circ \Ri_\Omega \circ P_1)(z)
= \frac{1}{d_1 p_1} \Ri_\Omega(p_1 z + p_0) + a_1,
\end{equation}
and
\begin{equation}
L = \left\{ w \in \C : \abs{w - a_1} \leq \frac{1}{d_1 \abs{p_1}} \right\}.
\end{equation}
\end{remark}

Formula~\eqref{eqn:relation_between_maps} does not lead to \emph{separate} 
expressions for $Q$ and $\Phi$, even if $\Ri_\Omega$ and $P_n$ are known.
However, if the polynomial $Q(w) = d_1 p_n \prod_{j=1}^\ell (w-a_j)^{n_j}$ is 
known, equation~\eqref{eqn:formal_expression_Phi}
yields an expression for $\Phi$.
Since the numbers $n_j$ are already known (Theorem~\ref{thm:E_decomposition}),
our next aim is to determine $a_1, \ldots, a_\ell$.

% \subsection{Critical points and critical values}

\begin{lemma} \label{lem:critpts}
Let $E = P_n^{-1}(\Omega)$ be as in~\eqref{eqn:E_poly_preimage} and with $\ell$ 
components.
\begin{enumerate}
% \item \label{it:critpts_1}
% The functions $Q \circ \Phi = \Ri_\Omega \circ P_n$ and $P_n$ have the same 
% critical points in $\C \setminus E$.  The total multiplicity of these 
% critical points is $\ell-1$.

\item \label{it:critpts_2}
A point $z_* \in \C \setminus E$ is a critical point of $P_n$ if and 
only if $w_* = \Phi(z_*)$ is a critical point of $Q$ in $\C \setminus L$.
Moreover, in that case
\begin{equation} \label{eqn:relation_critvals}
Q(w_*) = (\Ri_\Omega \circ P_n)(z_*).
\end{equation}

\item \label{it:critpts_3}
The polynomial $Q$ has $\ell - 1$ critical points in $\C \setminus L$ and these 
are the zeros of
\begin{equation} \label{eqn:critpts_Q}
\sum_{k=1}^\ell n_k \prod_{j=1, j \neq k}^\ell (w-a_j).
\end{equation}
\end{enumerate}
\end{lemma}

\begin{proof}
\ref{it:critpts_2}
By Theorem~\ref{thm:E_decomposition}, $P_n$ has $\ell-1$ critical points in $\C 
\setminus E$.
The functions $P_n$ and $\Ri_\Omega \circ P_n$ have the same critical points in 
$\comp{E}$ since $\Ri_\Omega$ is conformal in $\comp{\Omega}$ and
$(\Ri_\Omega \circ P_n)'(z) = \Ri_\Omega'(P_n(z)) P_n'(z)$.
By Theorem~\ref{thm:relation_between_maps}, we have $Q \circ \Phi = 
\Ri_\Omega \circ P_n$ in $\comp{E}$.
Since $(Q \circ \Phi)'(z) = Q'(\Phi(z)) \Phi'(z)$ and $\Phi$ is conformal,
we conclude that $z_*$ is a critical point of $Q \circ \Phi$ if and only if 
$w_* = \Phi(z_*)$ is a critical point of $Q$ which 
gives~\eqref{eqn:relation_critvals}.

\ref{it:critpts_3}
By~\ref{it:critpts_2}, $Q$ has exactly $\ell-1$ critical points in $\comp{L}$.  
By~\eqref{eqn:Q},
\begin{equation}
Q'(w) = d_1 p_n \prod_{j=1}^\ell (w-a_j)^{n_j-1} \cdot \bigg( \sum_{k=1}^\ell 
n_k \prod_{j=1, j \neq k}^\ell (w-a_j) \bigg),
\end{equation}
hence $a_1, \ldots, a_\ell$ are critical points of $Q$ with multiplicity
$\sum_{j=1}^\ell (n_j-1) = n-\ell$.
The remaining $\ell - 1$ critical points of $Q$ are the zeros of 
the polynomial in~\eqref{eqn:critpts_Q}.
\end{proof}

In principle, the right hand side in~\eqref{eqn:relation_critvals} can be 
computed when $P_n$ and $\Ri_\Omega$ are given.  If also $Q(w_*)$ 
can be computed, \eqref{eqn:relation_critvals} yields $\ell-1$ equations for 
$a_1, \ldots, a_\ell$.

With the results that we have established, we obtain the conformal map onto a 
lemniscatic domain of polynomial pre-images under $P_n(z) = \alpha (z-\beta)^n 
+ \gamma$, and of pre-images with one component ($\ell = 1$) and arbitrary 
polynomial.

\begin{proposition} \label{prop:example_zn}
Let $\Omega \subseteq \C$ be a simply connected infinite compact set.
Let $P_n(z) = \alpha (z-\beta)^n + \gamma$ with $\alpha, \beta, \gamma \in \C$, 
$\alpha \neq 0$, and $n \geq 2$.
\begin{enumerate}
\item \label{it:example_zn_1}
If $\gamma \notin \Omega$ then $E = P_n^{-1}(\Omega)$ has $n$ components,
$m_j = 1/n$ for $j = 1, \ldots, n$, the points $a_1, \ldots, a_n$ are given by
\begin{equation*}
a_{1, \ldots, n}
= \beta + \sqrt[n]{-\frac{\Ri_\Omega(\gamma)}{d_1 \alpha}}
\end{equation*}
with the $n$ distinct values of the $n$-th root and $d_1 = \Ri_\Omega'(\infty) 
> 0$,
\begin{equation} \label{eqn:example_zn_L}
L
% = \{ w \in \C : \abs{d_1 \alpha (w - \beta)^n + \Ri_\Omega(\gamma)} \leq 1 \},
= \Big\{ w \in \C : \prod_{j=1}^n \abs{w-a_j}^{1/n}
= \abs*{(w-\beta)^n + \frac{\Ri_\Omega(\gamma)}{d_1 \alpha}}^{1/n}
\leq (d_1 \abs{\alpha})^{-1/n} \Big\},
\end{equation}
and the Walsh map is
\begin{equation} \label{eqn:example_zn_Phi}
\Phi : \comp{E} \to \comp{L}, \quad
\Phi(z) = \beta + \sqrt[n]{\frac{\Ri_\Omega(P_n(z)) - 
\Ri_\Omega(\gamma)}{d_1 \alpha}},
\end{equation}
with that branch of the $n$-th root such that $\Phi(z) = z + \cO(1/z)$ at 
infinity.
% with the principal branch of the $n$-th root.
% At $z = \beta$, the map $\Phi$ has a removable singularity with 
% $\Phi(\beta) = \beta$.

\item \label{it:example_zn_2}
If $\gamma \in \Omega$ then $E = P_n^{-1}(\Omega)$ has one component,
$L$ is the disk
\begin{equation} \label{eqn:example_zn_2_L}
L = \{ w \in \C : \abs{w-\beta} \leq \capacity(E) = (d_1 \abs{\alpha})^{-1/n} 
\},
\end{equation}
and the conformal map of $\comp{E}$ onto a lemniscatic domain is
\begin{equation} \label{eqn:example_zn_2_Phi}
\Phi : \comp{E} \to \comp{L}, \quad
\Phi(z) = \beta + \sqrt[n]{\frac{\Ri_\Omega(P_n(z))}{d_1 \alpha}},
\end{equation}
with that branch of the $n$-th root such that $\Phi(z) = z + \cO(1/z)$ at 
infinity.
\end{enumerate}
\end{proposition}

\begin{proof}
\ref{it:example_zn_1}
Since $P_n(\beta) = \gamma$,
the assumption $\gamma \notin \Omega$ is equivalent to $\beta \notin E$.
The only critical point of $P_n$ is $z_* = \beta$ with multiplicity $n-1$, 
hence $E$ has $\ell = n$ components by Theorem~\ref{thm:E_decomposition}.
The point $\beta_1 \coloneq \Phi(\beta) \in \C$ is then a critical point of 
$Q$ of multiplicity $n-1$ by Lemma~\ref{lem:critpts} \ref{it:critpts_2}.
Therefore, $Q'$ is a constant multiple of $(w - \beta_1)^{n-1}$ and
\begin{equation*}
Q(w) = \alpha_1 (w - \beta_1)^n + \gamma_1, \quad \alpha_1, \gamma_1 \in \C.
\end{equation*}
Next, let us determine $\alpha_1, \beta_1, \gamma_1$ in terms of 
$\alpha, \beta, \gamma$.  We have
\begin{equation*}
\gamma_1 = Q(\beta_1) = Q(\Phi(\beta)) = \Ri_\Omega(P_n(\beta)) = 
\Ri_\Omega(\gamma).
\end{equation*}
By~\eqref{eqn:Q}, the leading coefficient of $Q$ is $\alpha_1 = d_1 \alpha \neq 
0$.
Since $\ell = n$, we have
\begin{equation} \label{eqn:example_zn_Q}
Q(w) = \alpha_1 (w-\beta_1)^n + \gamma_1 = \alpha_1 \prod_{j=1}^n (w-a_j)
\end{equation}
with distinct $a_1, \ldots, a_n \in \C$.
In particular, $n_j = 1$ for $j = 1, \ldots, n$.
Equating the coefficients of $w^{n-1}$ in~\eqref{eqn:example_zn_Q} and using 
Theorem~\ref{thm:aj_at_infty}, we obtain    
\begin{equation*}
n \beta_1 = \sum_{j=1}^n a_j = \sum_{j=1}^n n_j a_j = - \frac{p_{n-1}}{p_n} = - 
\frac{\alpha (- n \beta)}{\alpha} = n \beta,
\end{equation*}
i.e., $\beta_1 = \beta$.
By~\eqref{eqn:example_zn_Q},
\begin{equation*}
a_{1, \ldots, n} = \beta_1 + \sqrt[n]{-\frac{\gamma_1}{\alpha_1}}
= \beta + \sqrt[n]{-\frac{\Ri_\Omega(\gamma)}{d_1 \alpha}}
\end{equation*}
with the $n$ distinct values of the $n$-th root.
By~\eqref{eqn:Q}, we have $L = \{ w \in \C : \abs{Q(w)} \leq 1 \}$,
which is equivalent to~\eqref{eqn:example_zn_L}.
Then~\eqref{eqn:example_zn_Phi} follows from~\eqref{eqn:formal_expression_Phi}.

% Then $\Phi = Q^{-1} \circ \Ri_\Omega \circ P_n$ formally is
% \begin{equation*}
% \Phi(z)
% = \beta + \sqrt[n]{\frac{\Ri_\Omega(P_n(z)) - \Ri_\Omega(\gamma)}{d_1 
% \alpha}}
% = \beta + (z-\beta) \sqrt[n]{\frac{\Ri_\Omega(P_n(z)) - 
% \Ri_\Omega(\gamma)}{d_1 \alpha (z-\beta)^n}}.
% \end{equation*}
% The term under the $n$-th root behaves as $1 + \cO(1/z^n)$ for $z \to 
% \infty$.  Therefore, choosing the principal branch of the $n$-th root 
% yields~\eqref{eqn:example_zn_Phi},
% which satisfies $\Phi(z) = z + \cO(1/z)$ at infinity.

\ref{it:example_zn_2}
The assumption $\gamma \in \Omega$ is equivalent to $\beta \in E$, 
thus $P_n$ has no critical point in $\C \setminus E$ and $E$ is connected, 
i.e., $\ell = 1$.  Then $m_1 = 1$ and $n_1 = n$.
By Theorem~\ref{thm:relation_between_maps}, $n a_1 = - \frac{p_{n-1}}{p_n} = n 
\beta$, hence $a_1 = \beta$.  Together 
with~\eqref{eqn:capacity_poly_preimage}, we obtain the 
expression~\eqref{eqn:example_zn_2_L} for $L$.
In contrast to case (i), we have $Q(w) = d_1 \alpha (w-\beta)^n$, which 
yields~\eqref{eqn:example_zn_2_Phi} by~\eqref{eqn:formal_expression_Phi}.
% \begin{equation*}
% \Phi(z) = (Q^{-1} \circ \Ri_\Omega \circ P_n)(z) = \beta + (z-\beta) 
% \sqrt[n]{\frac{\Ri_\Omega(P_n(z))}{d_1 \alpha (z-\beta)^n}}.
% \end{equation*}
% The term under the $n$-th root behaves as $1 + \cO(1/z^n)$ for $z \to 
% \infty$.  Therefore, we choose the principal branch of the $n$-th root to 
% obtain $\Phi(z) = z + \cO(1/z)$.
\end{proof}

In~\cite[Thm.~3.1]{SeteLiesen2016}, the lemniscatic domain and conformal map 
$\Phi$ have been explicitly 
constructed under the additional assumptions that $\Omega$ is symmetric with 
respect to $\R$ (i.e., $\Omega^* = \Omega$), $\gamma \in \R$ is left of 
$\Omega$, $\alpha > 0$ and $\beta = 0$.  A shift $\beta \neq 0$ can be 
incorporated with~\cite[Lem.~2.3]{SeteLiesen2016}.
In Proposition~\ref{prop:example_zn} we can relax the assumptions on $\Omega$ 
and the coefficients $\alpha, \beta, \gamma$.

The proof of Proposition~\ref{prop:example_zn} \ref{it:example_zn_2} 
generalizes to arbitrary polynomials $P_n$ of degree $n \geq 2$, which yields 
the following result for a connected polynomial pre-image.

\begin{corollary} \label{cor:connected_pre-image}
Let $\Omega \subseteq \C$ be a simply connected infinite compact set.
Let $P_n$ be a polynomial of degree $n \geq 2$ as in~\eqref{eqn:Pn} such that 
$E = P_n^{-1}(\Omega)$ is connected, i.e., $\ell = 1$.  Then $L = \{ w \in \C : 
\abs{w-a_1} \leq (d_1 \abs{p_n})^{-1/n} \}$ with $m_1 = 1$ and $a_1 = - 
\frac{p_{n-1}}{n p_n}$, and
\begin{equation*}
\Phi : \comp{E} \to \comp{L}, \quad
\Phi(z) = a_1 + \sqrt[n]{\frac{\Ri_\Omega(P_n(z))}{d_1 p_n}},
\end{equation*}
with that branch of the $n$-th root such that $\Phi(z) = z + \cO(1/z)$ at 
infinity.
\end{corollary}

\begin{proof}
The assumption $\ell = 1$ implies $m_1 = 1$ and $n_1 = n$. By 
Theorem~\ref{thm:relation_between_maps}, we have $a_1 = - \frac{p_{n-1}}{n 
p_n}$, which yields the expressions for $L$, $Q(w) = d_1 p_n (w-a_1)^n$
and $\Phi$.
\end{proof}

Let us consider the case $\ell = 2$ in more detail.  In this case, $P_n$ has 
exactly one critical point outside $E$.

\begin{theorem} \label{thm:aj_for_two_components}
Let $E = P_n^{-1}(\Omega)$ in~\eqref{eqn:E_components} consist of two 
components, and let $z_*$ be the critical point of $P_n$ in $\C \setminus E$.
Then $a_1, a_2$ satisfy
\begin{align}
\left( a_2 + \frac{p_{n-1}}{n p_n} \right)^n &=
\frac{(-1)^{n_2} n_1^{n_2}}{d_1 p_n n_2^{n_2}}
(\Ri_\Omega \circ P_n)(z_*), \label{eqn:a2_for_2components} \\
a_1 &= -\frac{1}{n_1} \left( \frac{p_{n-1}}{p_n} + n_2 a_2 \right). 
\label{eqn:a1_for_2components}
\end{align}
\end{theorem}

\begin{proof}
By Theorem~\ref{thm:aj_at_infty}, the centers $a_1, a_2$ of $L$ satisfy 
\begin{equation} \label{eqn:lineq_aj_ell2}
n_1 a_1 + n_2 a_2 = - \frac{p_{n-1}}{p_n},
\end{equation}
or, equivalently,
\begin{equation} \label{eqn:poly_preimage_ell2_a1a2}
a_1 = -\frac{1}{n_1} \left( \frac{p_{n-1}}{p_n} + n_2 a_2 \right),
\quad \text{and} \quad
a_2 - a_1 = \frac{n}{n_1} \left( \frac{p_{n-1}}{n p_n} + a_2 \right).
\end{equation}
By Lemma~\ref{lem:critpts}~\ref{it:critpts_3}, the only critical 
point $w_*$ of $Q$ in $\C \setminus L$ is the zero of
$n_1 (w-a_2) + n_2 (w-a_1)$, i.e.,
\begin{equation*}
w_* = \frac{n_2 a_1 + n_1 a_2}{n}.
\end{equation*}
The corresponding critical value is
\begin{align*}
Q(w_*)
&= d_1 p_n (w_* - a_1)^{n_1} (w_* - a_2)^{n_2}
= d_1 p_n \left( \frac{n_1}{n} (a_2-a_1) \right)^{n_1} \left( 
\frac{n_2}{n} (a_1-a_2) \right)^{n_2} \\
&= d_1 p_n (-1)^{n_2} \frac{n_1^{n_1} n_2^{n_2}}{n^n} (a_2 - a_1)^n
= d_1 p_n (-1)^{n_2} \frac{n_2^{n_2}}{n_1^{n_2}} 
\left( a_2 + \frac{p_{n-1}}{n p_n} \right)^n,
\end{align*}
where we used~\eqref{eqn:poly_preimage_ell2_a1a2} in the last step.
Since $(\Ri_\Omega \circ P_n)(z_*) = Q(w_*)$ by 
Lemma~\ref{lem:critpts}\,\ref{it:critpts_2}, 
formula~\eqref{eqn:a2_for_2components} is established.
\end{proof}

% \begin{remark}
% If $\Omega$ is not further specified, it is not clear which branch of the 
% $n$-th root has to be taken in~\eqref{eqn:a2_for_2components} in order to 
% calculate $a_2$.
% \end{remark}

% Equation~\eqref{eqn:a2_for_2components} has $n$ possible solutions for $a_2$.
In order to specify the branch of the $n$-th root 
in~\eqref{eqn:a2_for_2components}, some additional information is needed.
We show this for a set $\Omega$ which is symmetric with respect to the real 
axis and contains the origin, which covers the important examples $\Omega = 
\overline{\bD}$ and $\Omega = \cc{-1, 1}$.

\begin{lemma} \label{lem:sign_Pn_at_crit_pts}
Suppose that $\Omega^* = \Omega$ and $0 \in \Omega$.
Let $P_n$ be a polynomial of degree $n$ as in~\eqref{eqn:Pn}
with real coefficients
such that $P_n^{-1}(\Omega) = \cup_{j=1}^\ell E_j$ with $E_j^* = E_j$ 
for $j = 1, \ldots, \ell$.
Denote the critical points of $P_n$ in $\C \setminus E$ by $z_1, \ldots, 
z_{\ell-1}$.
\begin{enumerate}
\item Then $z_1, z_2, \ldots, z_{\ell-1} \in \R$ and $z_j$ is between $E_j \cap 
\R$ and $E_{j+1} \cap \R$ for each $j = 1, \ldots, \ell-1$, where we label 
$E_1, \ldots, E_\ell$ from left to right along the real line.

\item For each $j \in \{ 1, \ldots, \ell-1 \}$ and each $z \in \oo{\max(E_j 
\cap \R), \min(E_{j+1} \cap \R)}$, we have 
\begin{equation}
\sgn(\Ri_\Omega(P_n(z))) = \sgn(P_n(z)) = (-1)^{n_{j+1} + \ldots + n_\ell} 
\sgn(p_n),
\end{equation}
which holds in particular for $z = z_j$.
\end{enumerate}
If $\Omega$ is additionally symmetric with respect to the imaginary axis, 
the assertions also hold if $P_n$ has purely imaginary coefficients.
\end{lemma}

\begin{proof}
(i) Note that $P_n$ and $g_E$ have the same critical points in $\comp{E}$, 
compare the proof of Theorem~\ref{thm:E_decomposition}.  Then, since $E_j^* = 
E_j$, (i) is a special case of Theorem~\ref{thm:crit_pts_of_gE}.

(ii) Since $\Omega^* = \Omega$, the Riemann map satisfies $\Ri_\Omega(z) = 
\conj{\Ri_\Omega(\conj{z})}$ for $z \in \Omega$.
In particular, if $z \in \R \setminus \Omega$, also $\Ri_\Omega(z) \in \R$.  
Together with $\Ri_\Omega'(z) > 0$, we have that
$\Ri_\Omega(\oo{\max(\Omega \cap \R), \infty}) = \oo{1, \infty}$ and 
$\Ri_\Omega(\oo{- \infty, \min(\Omega \cap \R)}) = \oo{-\infty, 1}$.
Since $0 \in \Omega$, we see that $\sgn(\Ri_\Omega(z)) = \sgn(z)$ for $z \in \R 
\setminus \Omega$.

Similarly, if $\Omega$ is additionally symmetric with respect to the imaginary 
axis, $\Ri_\Omega$ maps the imaginary axis onto itself and
$\sgn(\Ri_\Omega(z)) = \sgn(z)$ for $z \in (i \R) \setminus \Omega$.

We can treat the cases that the coefficients $P_n$ are real or purely imaginary 
(provided that $\Omega$ is also symmetric with respect to the imaginary axis) 
together.
If $z \in \R \setminus E$, we have $P_n(z) \in \R \setminus \Omega$ (or $P_n(z) 
\in (i \R) \setminus \Omega$) and hence $\sgn(\Ri_\Omega(P_n(z))) = 
\sgn(P_n(z))$.
It remains to compute $\sgn(P_n(z))$.  Since $0 \in \Omega$, we have 
$\sgn(P_n(z)) = \sgn(p_n)$ for $z > \max(E_\ell \cap \R)$.
Moreover, $P_n$ has $n_\ell$ zeros in $E_\ell$ which are either real or appear 
in complex conjugate pairs.  Therefore $\sgn(P_n(z)) = (-1)^{n_\ell} \sgn(p_n)$ 
for $z$ in the rightmost gap, 
i.e., $z \in \oo{\max(E_{\ell-1} \cap \R), \min(E_\ell \cap \R)}$.
Similarly, we get the assertion for the next gap and so on.
\end{proof}

\begin{corollary} \label{cor:aj_for_two_components}
Suppose that $\Omega^* = \Omega$ and $0 \in \Omega$.
Let $P_n$ be a polynomial of degree $n$ as in~\eqref{eqn:Pn}
with real coefficients such that 
$P_n^{-1}(\Omega) = E_1 \cup E_2$ with $E_1^* = E_1$ and $E_2^* = E_2$.
Let $n_1, n_2$ be the number of zeros of $P_n$ in $E_1$, $E_2$, respectively, 
and let $z_*$ be the critical point of $P_n$ in $\C \setminus E$.
Then the points $a_1, a_2$ are real with $a_1 < a_2$ and are given by
\begin{align}
a_1 &= - \frac{p_{n-1}}{n p_n} - \bigg( \left( \frac{n_2}{n_1} \right)^{n_1} 
\frac{(-1)^{n_2}}{d_1 p_n} \Ri_\Omega(P_n(z_*)) \bigg)^{\frac{1}{n}},
% &= - \frac{p_{n-1}}{n p_n} - \left( \frac{n_2}{n_1} \right)^{\frac{n_1}{n}} 
% \bigg( \frac{(-1)^{n_2}}{d_1 p_n} \Ri_\Omega(P_n(z_*)) \bigg)^{\frac{1}{n}},
% &= - \frac{p_{n-1}}{n p_n} - \frac{n_2}{n_1} \bigg( \frac{n_1^{n_2}}{d_1 
% p_n n_2^{n_2}} (-1)^{n_2} \Ri_\Omega(P_n(z_*)) \bigg)^{\frac{1}{n}},
% -\frac{1}{n_1} \left( \frac{p_{n-1}}{p_n} + n_2 a_2 \right),
\label{eqn:a1_for_two_intervals} \\
% 
% a_2 &= - \frac{p_{n-1}}{n p_n} + \bigg( \frac{n_1^{n_2}}{d_1 p_n n_2^{n_2}} 
% (-1)^{n_2} \Ri_\Omega(P_n(z_*)) \bigg)^{\frac{1}{n}},
a_2 &= - \frac{p_{n-1}}{n p_n} +
\bigg( \left( \frac{n_1}{n_2} \right)^{n_2} \frac{(-1)^{n_2}}{d_1 p_n} 
\Ri_\Omega(P_n(z_*)) \bigg)^{\frac{1}{n}}, 
\label{eqn:a2_for_two_intervals}
\end{align}
with the positive real $n$-th root.
% in formula~\eqref{eqn:a2_for_two_intervals}.

If $\Omega$ is additionally symmetric with respect to the imaginary axis, then 
$P_n$ can also have purely imaginary coefficients.
\end{corollary}

\begin{proof}
By Theorem~\ref{thm:aj_real} and Theorem~\ref{thm:crit_pts_of_gE}, the 
points $a_1, a_2$ are real and $a_1 < a_2$.
By Theorem~\ref{thm:aj_for_two_components}, we 
have~\eqref{eqn:a2_for_2components},
% \begin{equation} \label{eqn:a2_for_two_components_intermediate_step}
% \left( a_2 + \frac{p_{n-1}}{n p_n} \right)^n
% = \frac{(-1)^{n_2} n_1^{n_2}}{d_1 p_n n_2^{n_2}} \Ri_\Omega(P_n(z_*)),
% \end{equation}
which gives~\eqref{eqn:a2_for_two_intervals}.
Since $\frac{(-1)^{n_2}}{p_n} \Ri_\Omega(P_n(z_*)) > 0$ by 
Lemma~\ref{lem:sign_Pn_at_crit_pts}\,(ii) and $d_1 > 0$, the right hand side in 
formula~\eqref{eqn:a2_for_2components} % two_components_intermediate_step}
is positive.
By~\eqref{eqn:a1_for_2components}, $a_1 < a_2$ is equivalent to $a_2 > 
- \frac{p_{n-1}}{n p_n}$, which shows that we have to take the positive real 
$n$-th root in~\eqref{eqn:a2_for_two_intervals}.
Inserting~\eqref{eqn:a2_for_two_intervals} into~\eqref{eqn:a1_for_2components} 
yields~\eqref{eqn:a1_for_two_intervals}.
\end{proof}

\section{Examples}
\label{sect:examples}

In this section, we consider six examples of polynomial pre-images 
$E = P_n^{-1}(\Omega)$ for the cases $\Omega = \cc{-1, 1}$, $\Omega = 
\overline{\bD}$ and $\Omega = \cE_R \coloneq \{ \frac{1}{2} (r e^{it} + 
r^{-1} e^{-it}) : t \in \co{0, 2 \pi}, 1 \leq r \leq R \}$ 
(Chebyshev ellipse), $R > 1$.  We have the exterior Riemann maps
\begin{equation*}
\Ri_{\cc{-1, 1}}(z) = z + \sqrt{z^2 - 1},
\quad \text{and} \quad
\Ri_{\cE_R}(z) = \frac{1}{R} (z + \sqrt{z^2 - 1}) = \frac{1}{R} \Ri_{\cc{-1, 
1}}(z),
\end{equation*}
where the branch of the square root is chosen such that 
$\abs{\Ri_{\cc{-1, 1}}(z)} > 1$.
In particular, the coefficients of $z$ at infinity are
$\Ri_{\cc{-1, 1}}'(\infty) = 2$ and $\Ri_{\cE_R}'(\infty) = 2/R$;
see~\eqref{eqn:riemannmap} and~\eqref{eqn:d1}.
We begin with three examples for Proposition~\ref{prop:example_zn}.

\begin{example}
\label{ex:zn_star}
Let $\Omega = \cc{-1, 1}$ and $P_n(z) = z^n$.  Since the critical value 
of $P_n$ is $0 \in \Omega$, the set $L$ and Walsh map $\Phi$ of the 
connected 
star $E = P_n^{-1}(\cc{-1, 1}) = \cup_{k=1}^n e^{k 2 \pi i/n} \cc{-1, 1}$ are 
given by Proposition~\ref{prop:example_zn} (ii) as
$L = \{ w \in \C : \abs{w} \leq 2^{-1/n} \}$ and
\begin{equation*}
\Phi : \comp{E} \to \comp{L}, \quad
\Phi(z)
= \sqrt[n]{\frac{z^n + \sqrt{z^{2n} - 1}}{2}}
= z \sqrt[n]{\frac{z^n + \sqrt{z^{2n} - 1}}{2 z^n}}.
\end{equation*}
We take the branch of the square root with $\abs{z^n + \sqrt{z^{2n} - 1}} > 1$.
In the second representation of $\Phi$ we take the principal branch of the 
$n$-th root; see~\cite[Thm.~3.1]{SeteLiesen2016}.
In particular, the logarithmic capacity of $E$ is $2^{-1/n}$.
Figure~\ref{fig:example_zn_star} illustrates the case $n = 5$.
The left panel shows a phase plot of $\Phi$.
In a phase plot, the domain is colored according to the phase $f(z)/\abs{f(z)}$ 
of the function $f$; see~\cite{Wegert2012} for an introduction to phase plots.
The middle and right panels show $E$ and $\partial L$ (in black) as well as a 
grid and its image under $\Phi$.

\begin{figure}
{\centering
\includegraphics[width=0.32\linewidth]{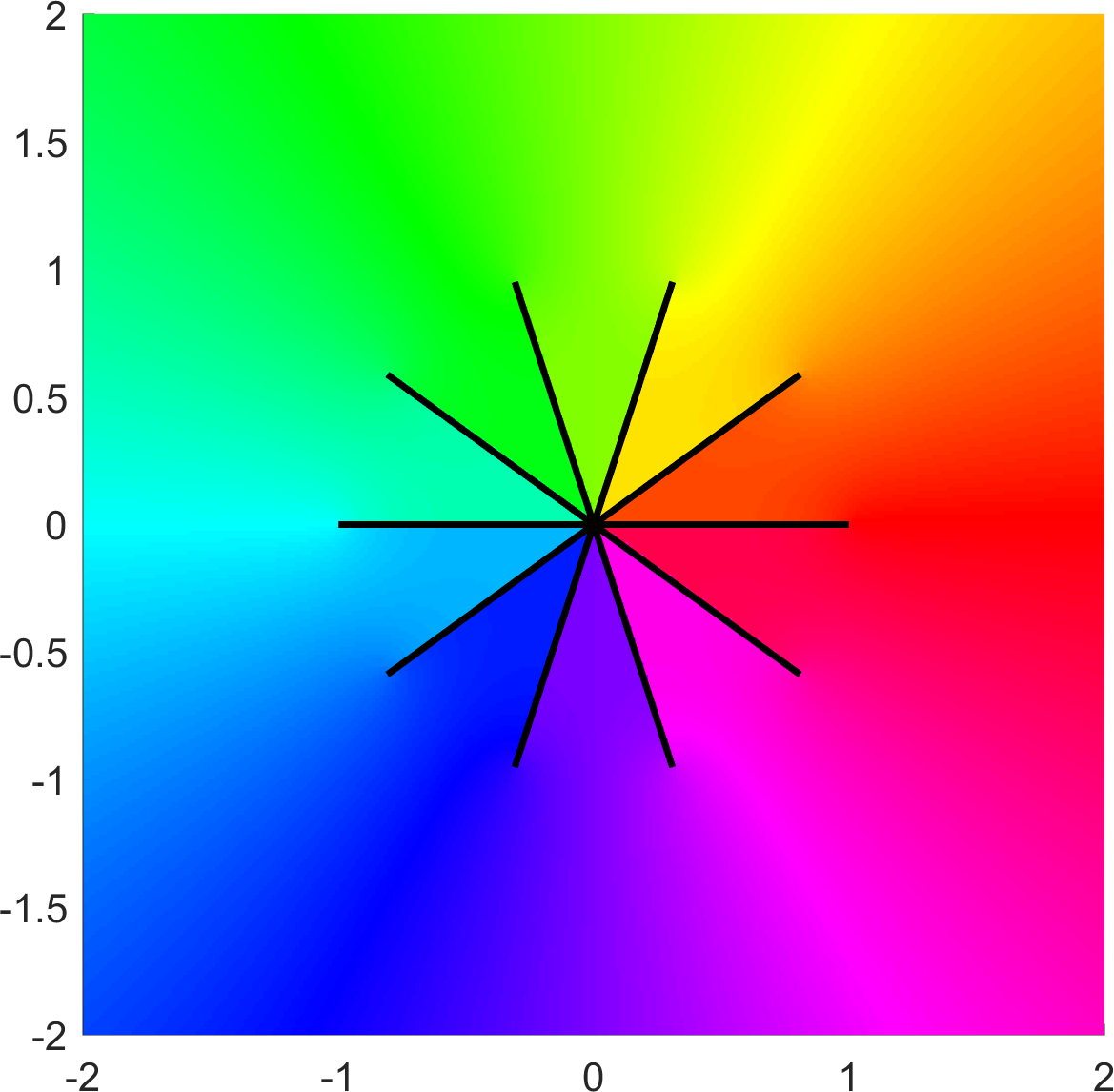}
\includegraphics[width=0.32\linewidth]{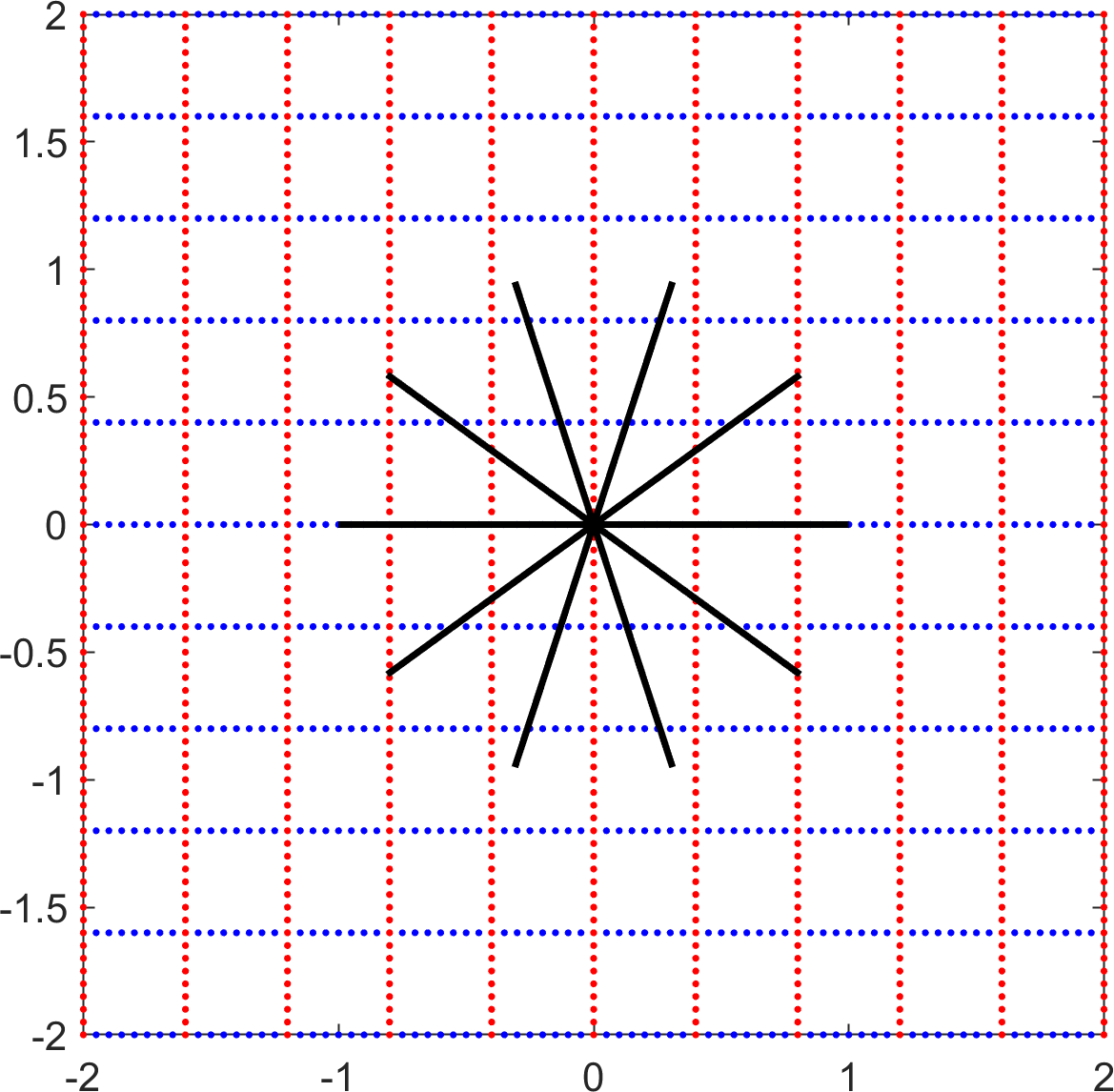}
\includegraphics[width=0.32\linewidth]{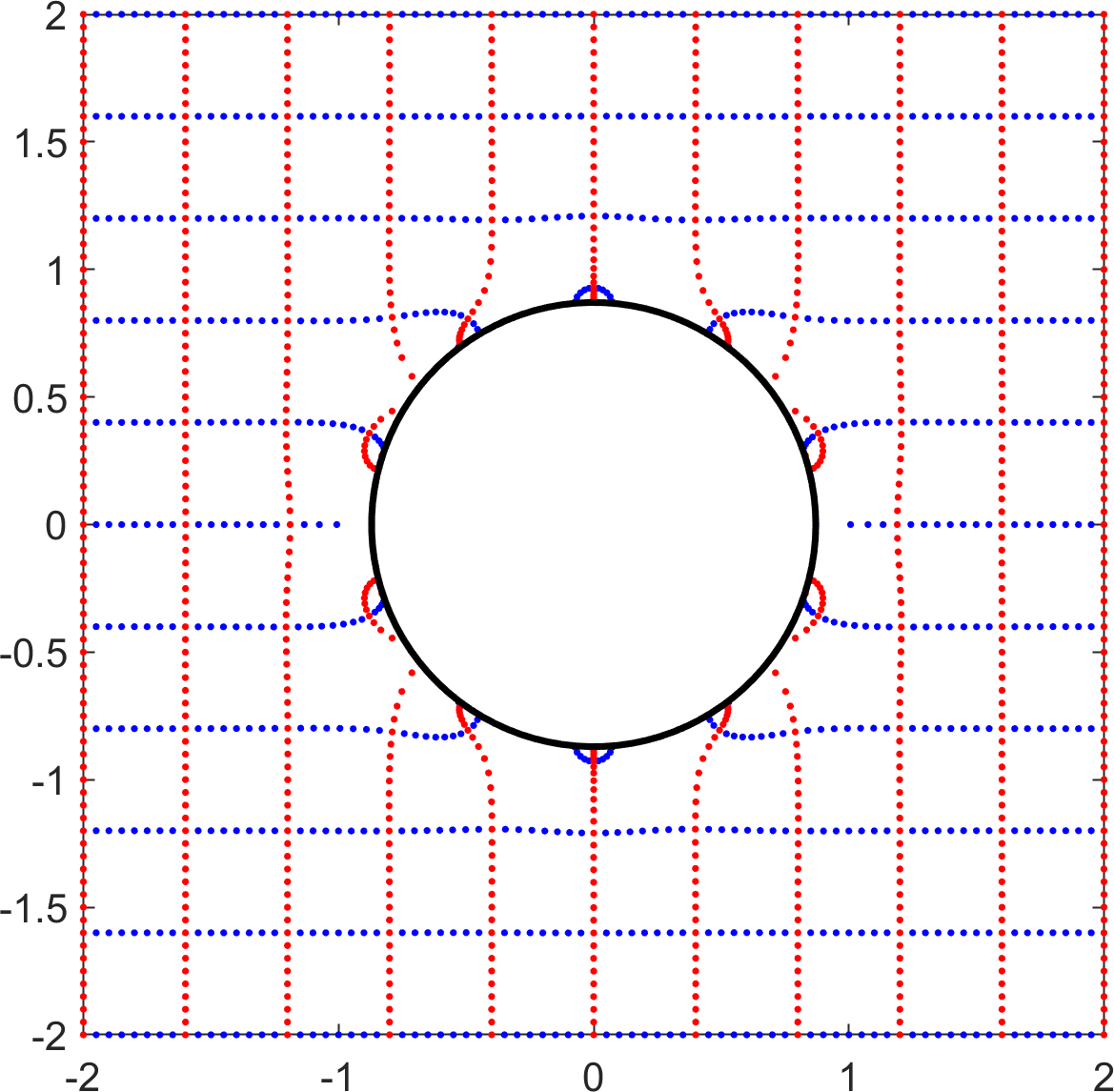}

}
\caption{Pre-image $E = P^{-1}(\cc{-1, 1})$ with $P(z) = z^5$ in 
Example~\ref{ex:zn_star}.
Left: Phase plot of $\Phi$, middle: $E$ (black) and grid, right: $\partial L$ 
(black) and image of the grid under $\Phi$.}
\label{fig:example_zn_star}
\end{figure}
\end{example}

\begin{example} \label{ex:zn_ellipse}
Let $\Omega = \cE_{1.25}$ be the Chebyshev ellipse bounded by $\{ \frac{1}{2} 
(\frac{5}{4} e^{it} + \frac{4}{5} e^{-it}) : t \in \co{0, 2 \pi} \}$
and let $E = P^{-1}(\cE_{1.25})$ with $P(z) = (z-1)^5 + \gamma$ for two 
different values of $\gamma$.
For $\gamma = 0.3i \notin \Omega$, the set $E$ consists of $n=5$ components, 
while for $\gamma = 0.75 \in \Omega$, the set $E$ has only one component; see
Proposition~\ref{prop:example_zn}.
Figure~\ref{fig:example_zn_ellipse} shows phase plots of $\Phi$ (left), the 
sets $\partial E$ and $\partial L$ in black and a grid and its image.
The phase plots show $\Phi$ and an analytic continuation to the interior of 
$E$.  The discontinuities in the phase (in the interior of $E$) are branch cuts 
of this analytic continuation.

\begin{figure}[t]
{\centering
\includegraphics[width=0.32\linewidth]{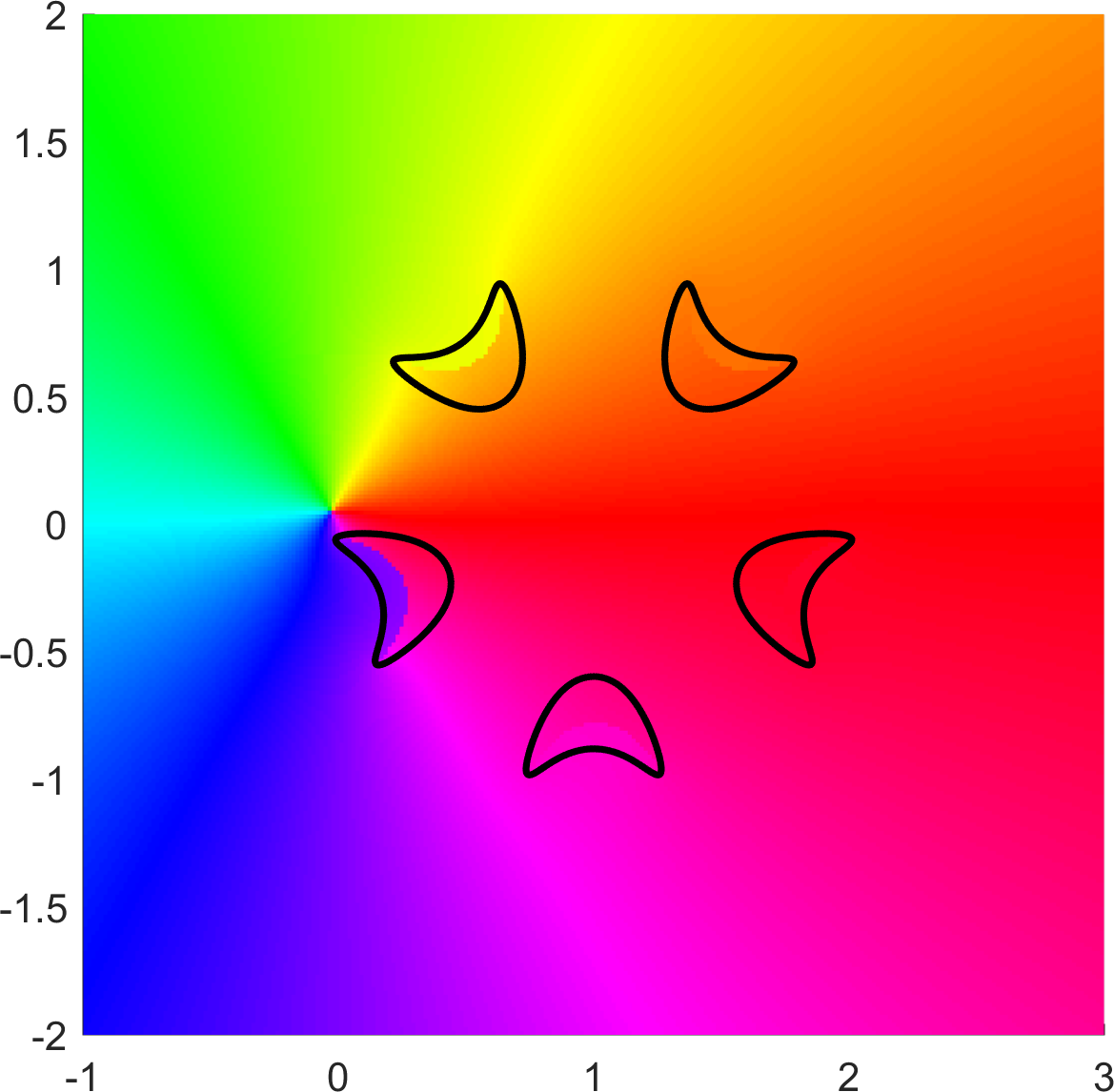}
\includegraphics[width=0.32\linewidth]{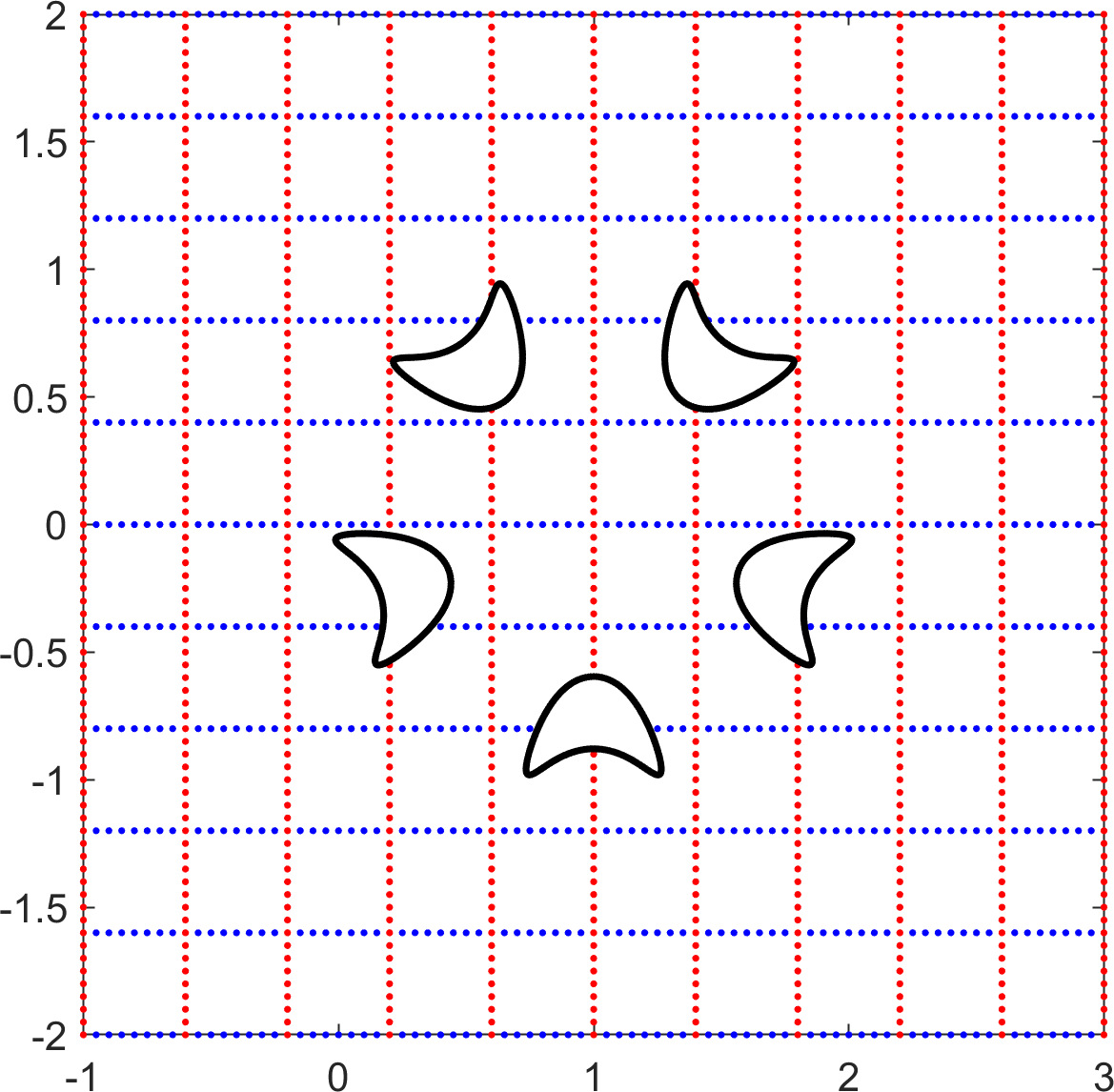}
\includegraphics[width=0.32\linewidth]{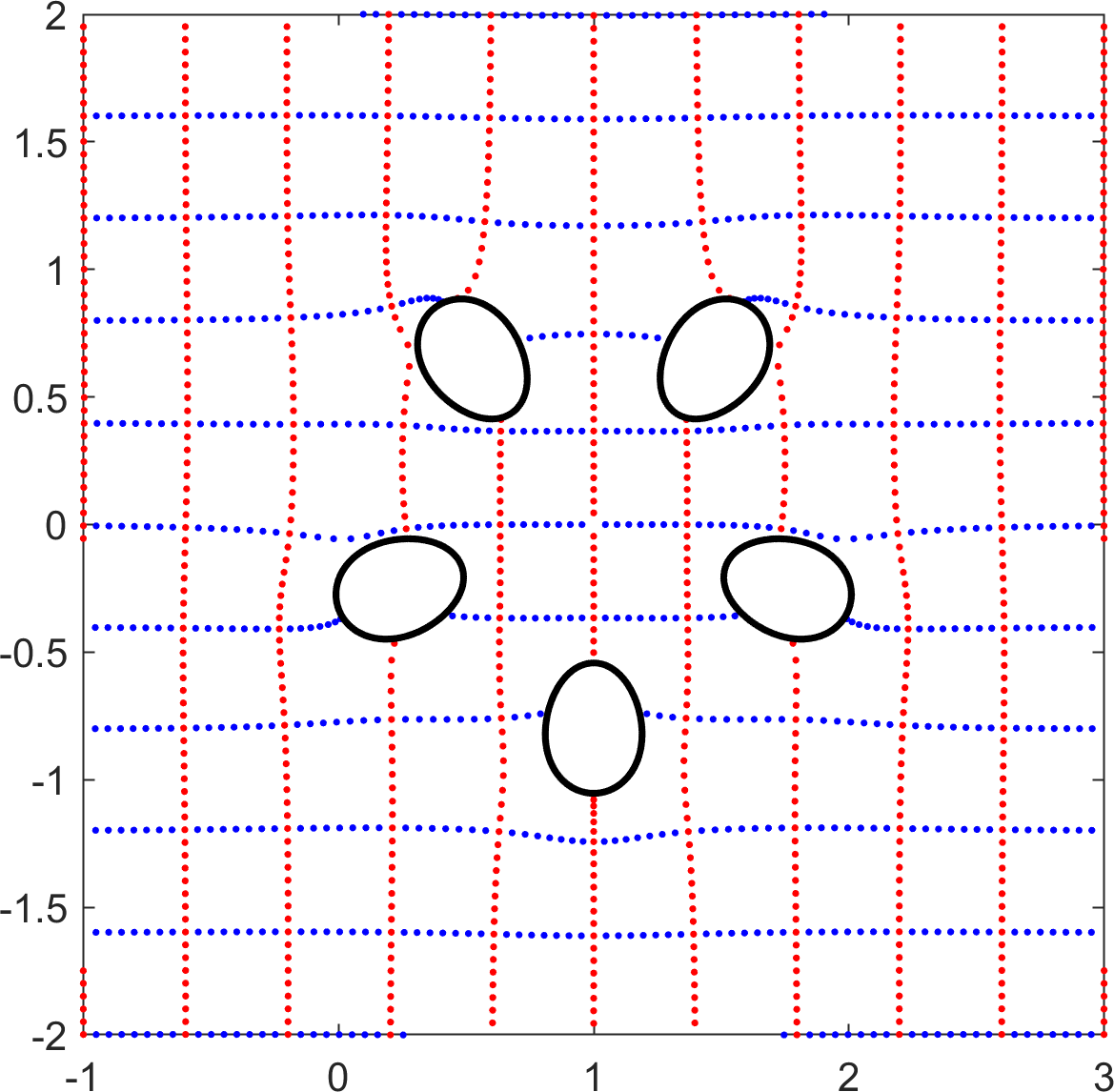}

\includegraphics[width=0.32\linewidth]{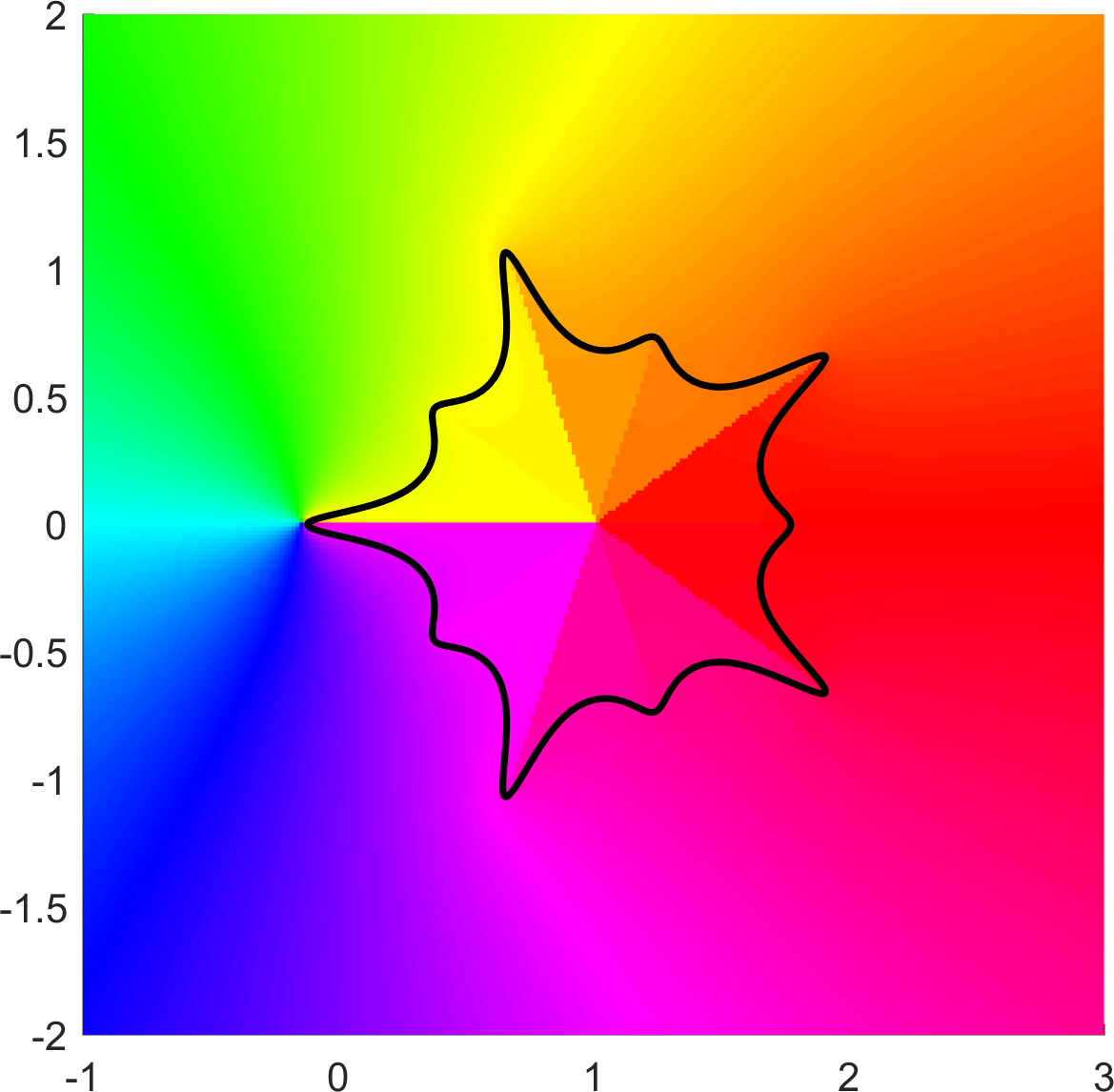}
\includegraphics[width=0.32\linewidth]{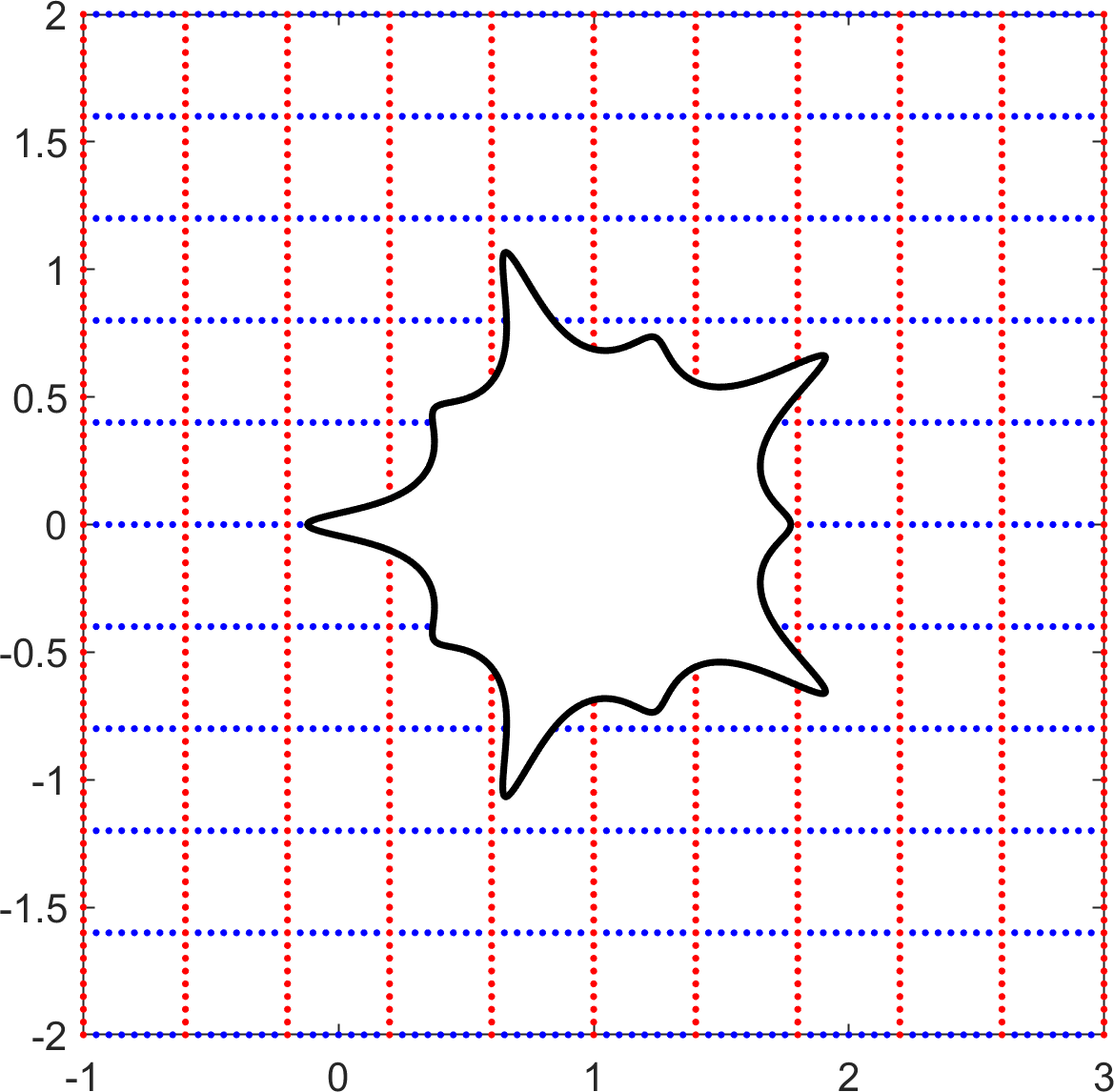}
\includegraphics[width=0.32\linewidth]{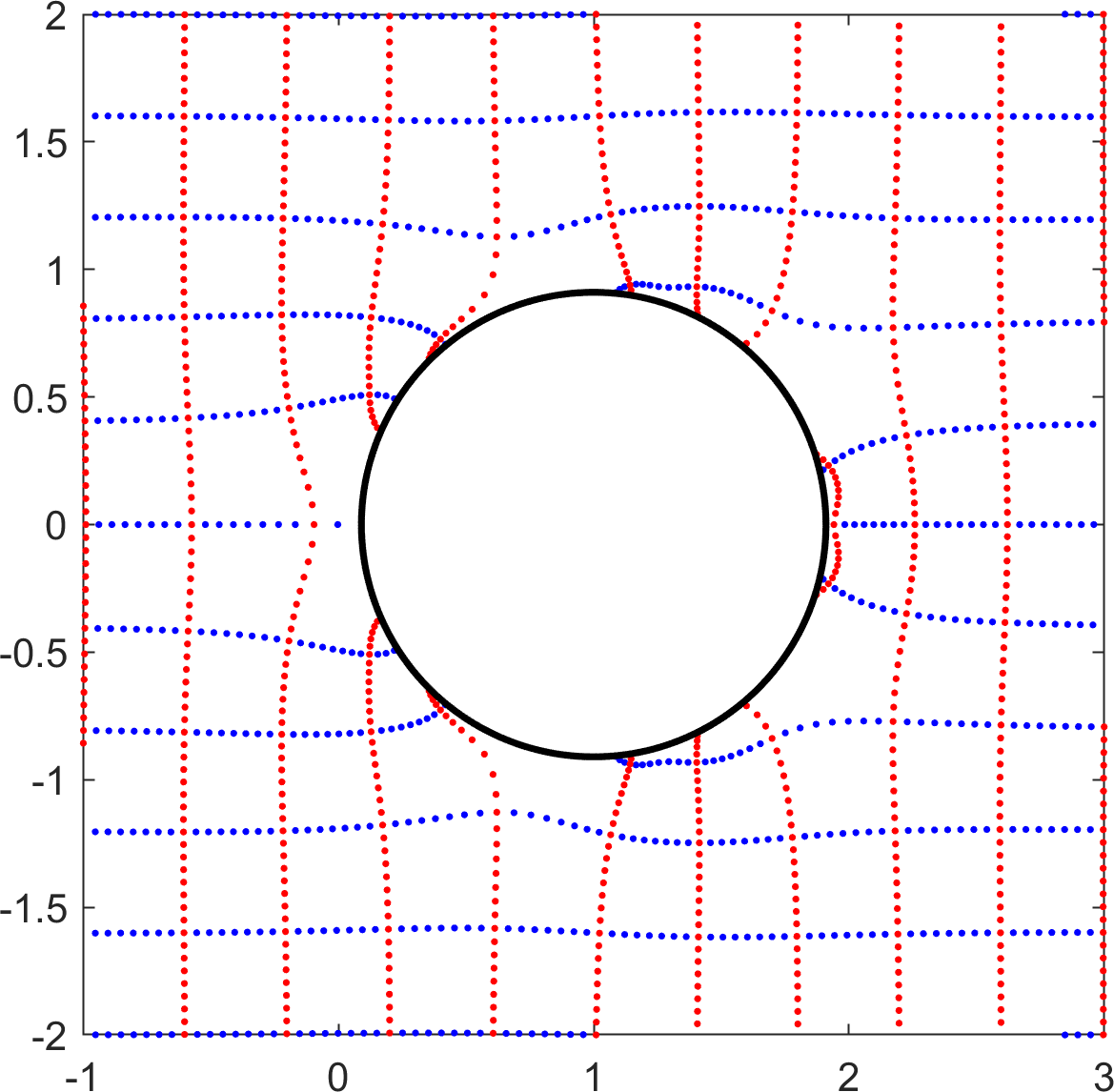}

}
\caption{Set $E = P_n^{-1}(\Omega)$ with a Chebyshev ellipse $\Omega = 
\cE_{1.25}$ and $P_n(z) = (z-1)^5 + \gamma$, with $\gamma = 0.3i \notin \Omega$ 
(top row) and $\gamma = 0.75 \in \Omega$ (bottom row); see 
Example~\ref{ex:zn_ellipse}.
Phase plot of $\Phi$ (left), original and image domains with $\partial E$ 
and $\partial L$ in black (middle and right).}
\label{fig:example_zn_ellipse}
\end{figure}
\end{example}

\begin{example} \label{ex:zn_disk}
Let $P_n(z) = \alpha (z-\beta)^n + \gamma$ and $E = P_n^{-1}(\overline{\bD})$.
\begin{enumerate}
\item If $\gamma \notin \overline{\bD}$ then $\Phi(z) = z$ by 
Proposition~\ref{prop:example_zn}, hence $L = E = \{ z \in \C : \abs{P_n(z)} 
\leq 1 \}$, i.e., $\comp{E}$ is a lemniscatic domain; see also 
Example~\ref{ex:preim_of_unit_disk_with_n_components} for pre-images of 
$\overline{\bD}$ under general polynomials.

\item If $\gamma \in \overline{\bD}$ then $E$ has only one component.  In this 
case $\Phi(z) = z$ if and only if $\gamma = 0$.
Figure~\ref{fig:example_zn_disk} shows an example with $\gamma \in 
\overline{\bD} \setminus \{ 0 \}$, where $E$ is not a lemniscatic domain and 
$\Phi(z) \neq z$.
\end{enumerate}

\begin{figure}[t]
{\centering
\includegraphics[width=0.32\linewidth]{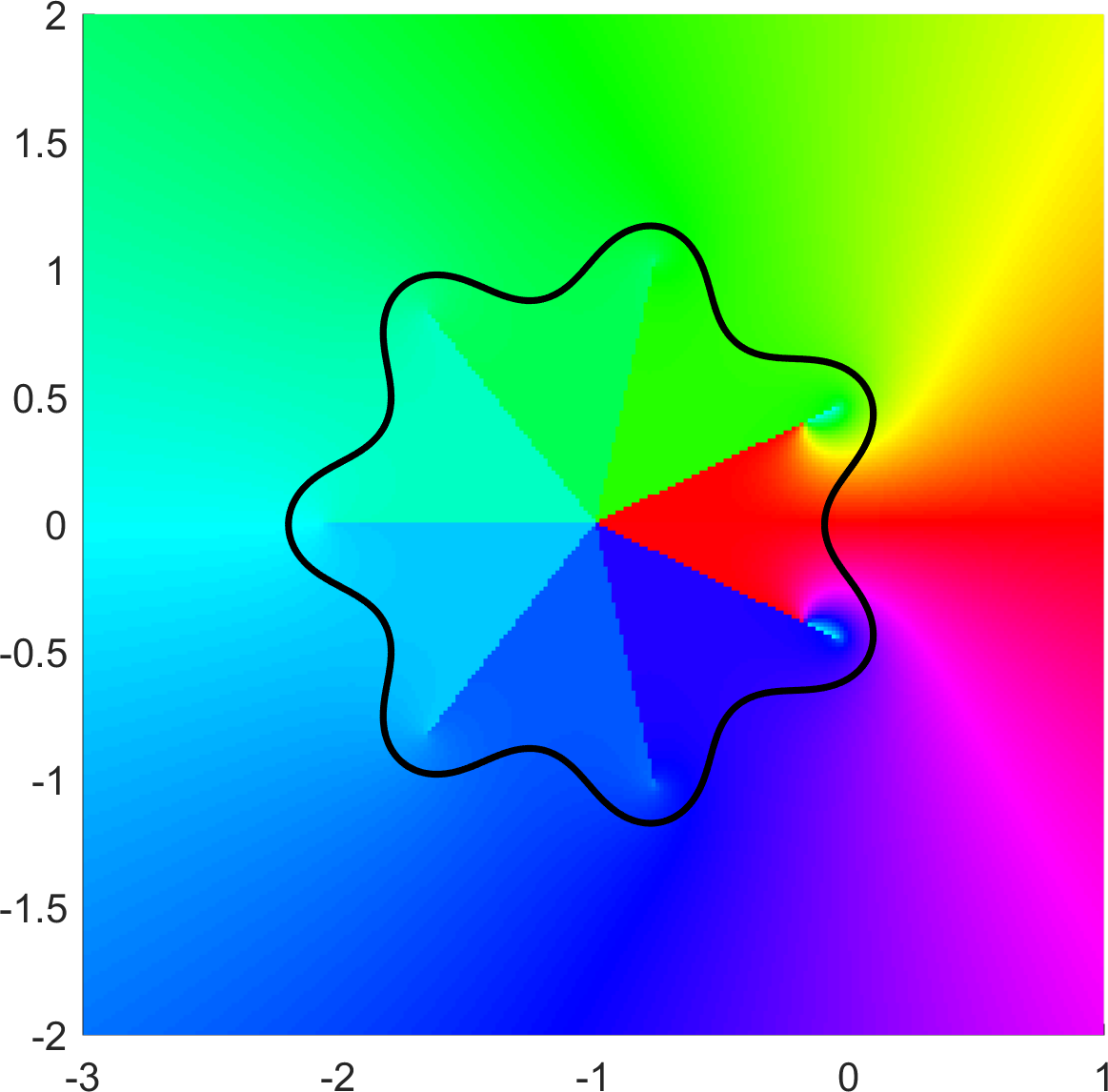}
\includegraphics[width=0.32\linewidth]{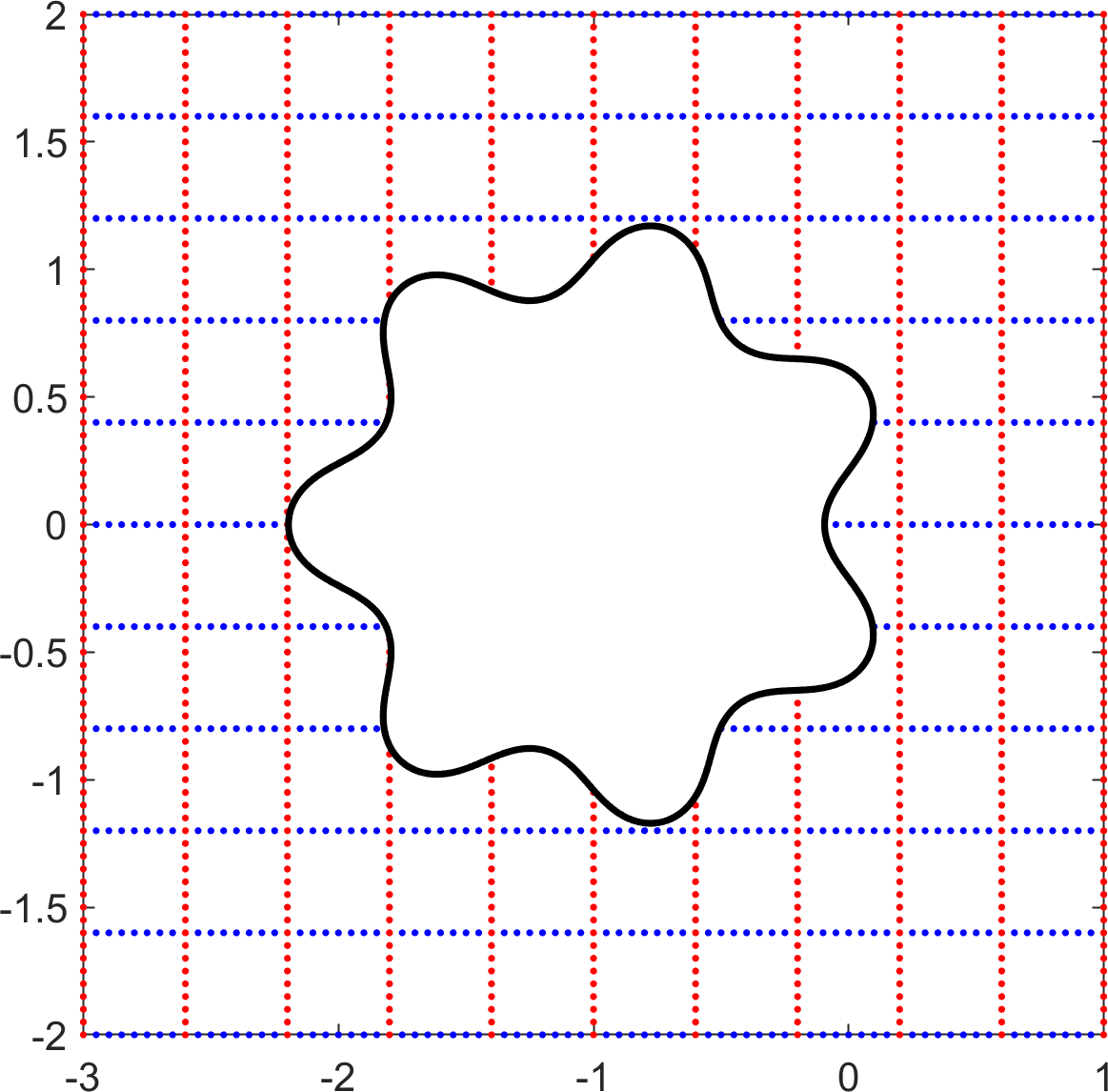}
\includegraphics[width=0.32\linewidth]{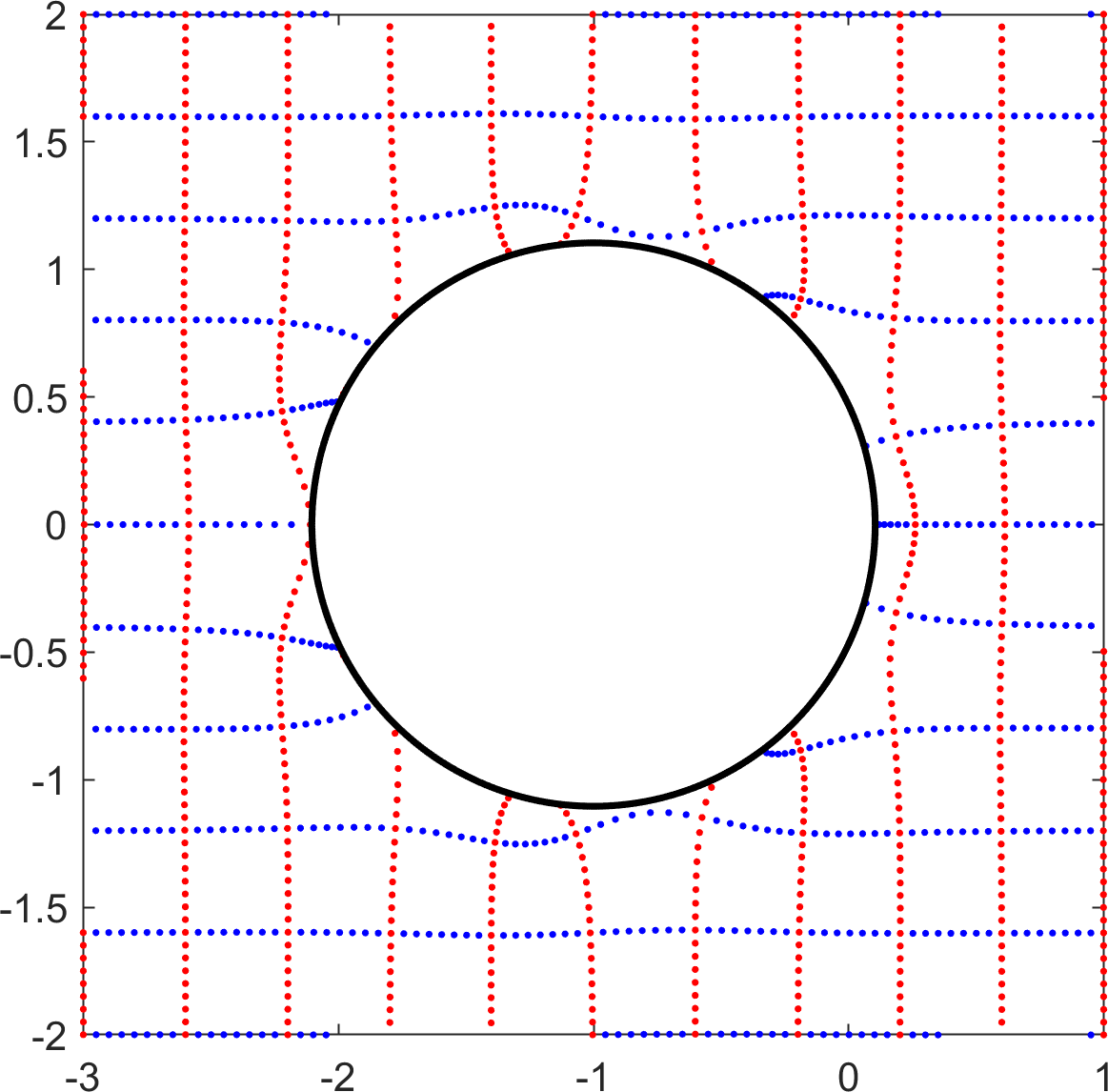}

}
\caption{The set $E = P_n^{-1}(\overline{\bD})$ with $P_n(z) = \frac{1}{2} 
(z+1)^7 + \frac{3}{4}$. Phase plot of $\Phi$ (left), original and image domains 
with $\partial E$ and $\partial L$ in black (middle and right); see 
Example~\ref{ex:zn_disk}.}
\label{fig:example_zn_disk}
\end{figure}
\end{example}

Next, we present an example for Corollary~\ref{cor:connected_pre-image}.

\begin{example} \label{ex:connected_pre-image}
Consider the polynomial
\begin{equation*}
P_4(z) = \frac{8 z^4 - 8 z^2 + \alpha}{\alpha}
\end{equation*}
with $\alpha \geq 1$ from~\cite[Example~(iv)]{Schiefermayr2014}.  Then $E = 
P_4^{-1}(\cc{-1, 1})$ is connected, since the 
critical points of $P_4$ are $0, \pm 1/\sqrt{2}$ with corresponding critical 
values $P_4(0) = 1 \in \cc{-1, 1}$ and $P_4(\pm 1/\sqrt{2}) = 1 - 
\frac{2}{\alpha} \in \cc{-1, 1}$; see Theorem~\ref{thm:E_decomposition}.
By Corollary~\ref{cor:connected_pre-image},
\begin{equation*}
L = \Bigg\{ w \in \C : \abs{w} \leq \capacity(E) = \frac{\alpha^{1/4}}{2} 
\bigg\},
\end{equation*}
and the conformal map is
\begin{equation*}
\Phi : \comp{E} \to \comp{L}, \quad
\Phi(z) = 
\frac{\sqrt[4]{\alpha}}{2} \sqrt[4]{P_4(z) + \sqrt{P_4(z)^2 - 1}},
\end{equation*}
see Figure~\ref{fig:connected_pre-image}.
Since $E^* = E$, we have that $\Phi(z) = \conj{\Phi(\conj{z})}$ and, since 
$\Phi(z) = z + \cO(1/z)$, we have in particular $\Phi(\oo{1, \infty}) = 
\oo{\capacity(E), \infty}$ and $\Phi(\oo{-\infty, -1}) = \oo{-\infty, - 
\capacity(E)}$.
Since $E$ is also symmetric with respect to the imaginary axis, we similarly 
have $\Phi(\oo{0, i \infty}) = \oo{i \capacity(E), i \infty}$ and 
$\Phi(\oo{-i \infty, 0}) = \oo{-i \infty, -i \capacity(E)}$.
Hence, $\Phi$ maps each quadrant to itself.
We use this to determine the correct branch of the fourth root.
\end{example}

\begin{figure}
{\centering
\includegraphics[width=0.32\linewidth]{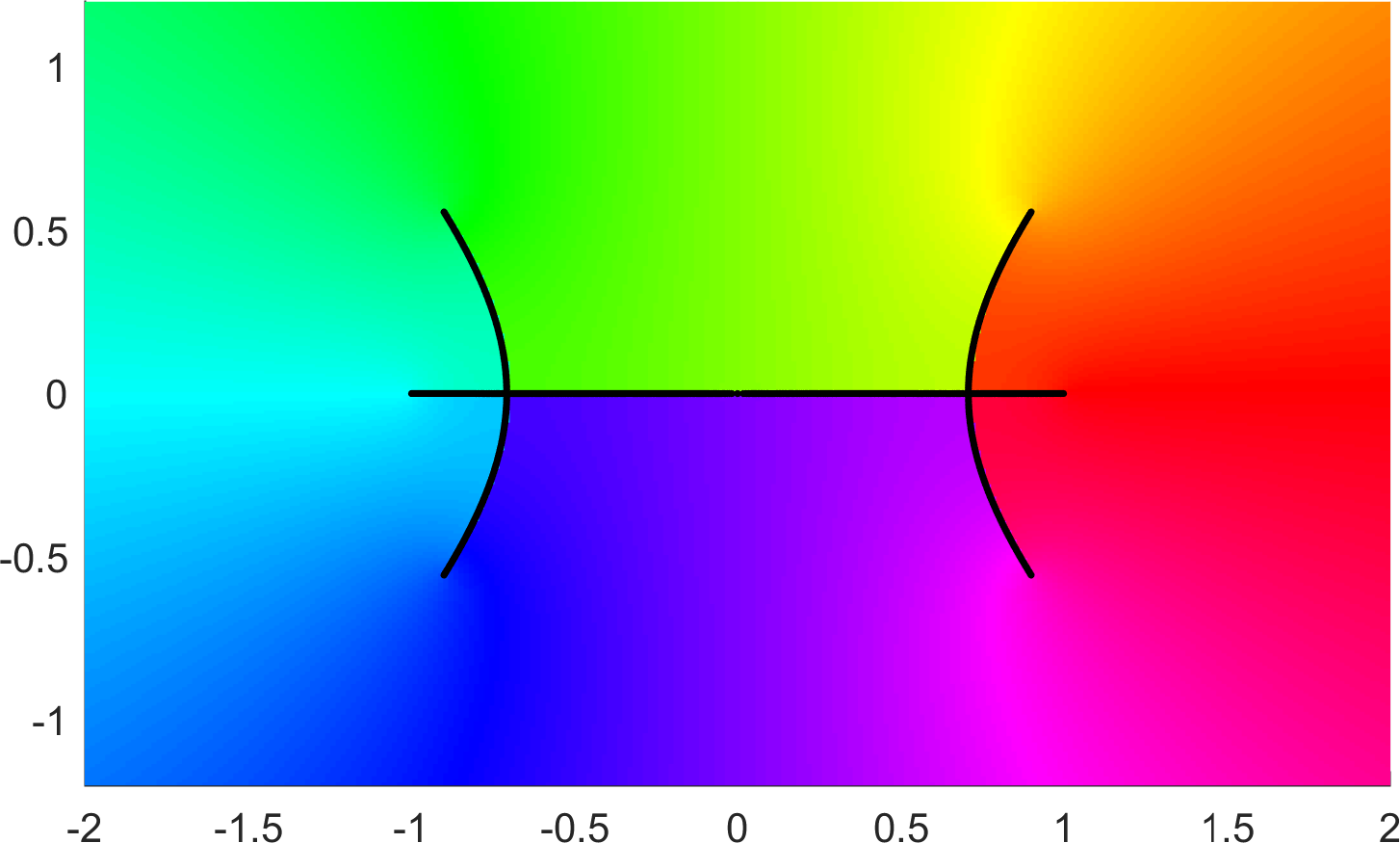}
\includegraphics[width=0.32\linewidth]{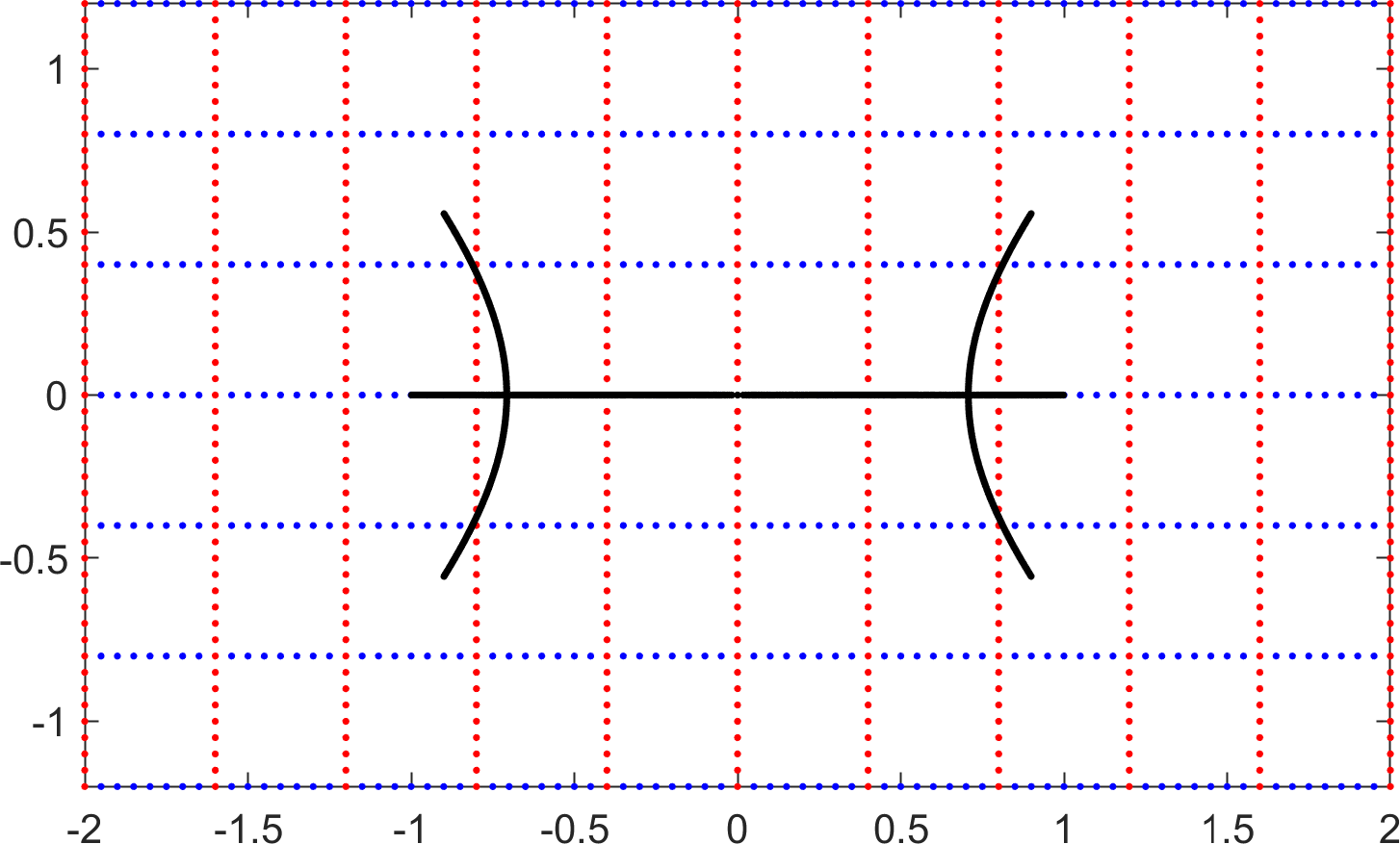}
\includegraphics[width=0.32\linewidth]{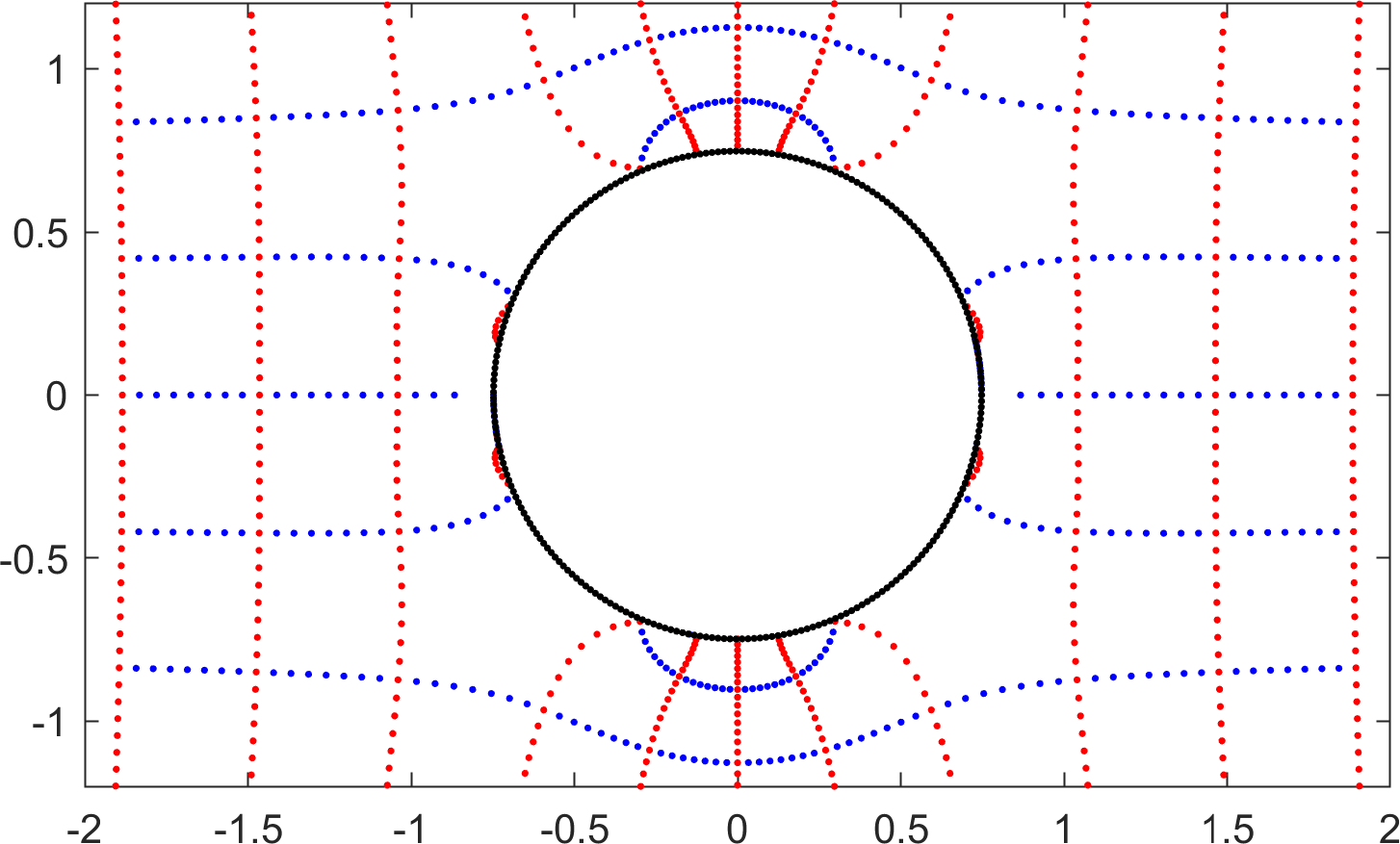}

}
\caption{The set $E = P_4^{-1}(\cc{-1, 1})$ with $\alpha = 5$ in 
Example~\ref{ex:connected_pre-image}.  Phase plot of $\Phi$ (left), original 
and 
image domains with $E$ and $\partial L$ in black (middle and right).}
\label{fig:connected_pre-image}
\end{figure}

\begin{example} \label{ex:preim_of_unit_disk_with_n_components}
Let $P_n(z) = p_n \prod_{j=1}^n (z-b_j)$ be a polynomial of degree $n$.
If $E = P_n^{-1}(\overline{\bD})$ consists of $n$ components then $\comp{E}$ is 
a lemniscatic domain, i.e., $L = E$ with $a_j = b_j$, $m_j = 1/n$, 
$\capacity(E) = \abs{p_n}^{-1/n}$, and $\Phi(z) = z$.
Similarly, if $P_n(z) = p_n \prod_{j=1}^\ell (z - b_j)^{n_j}$ with 
distinct $b_1, \ldots, b_\ell \in \C$ and if $E$ has $\ell$ components, 
then $\comp{E}$ is a lemniscatic domain, $L = E$ with $a_j = b_j$, $m_j = 
n_j/n$, and $\Phi(z) = z$.
\end{example}

\begin{figure}
{\centering
\includegraphics[width=0.4\linewidth]{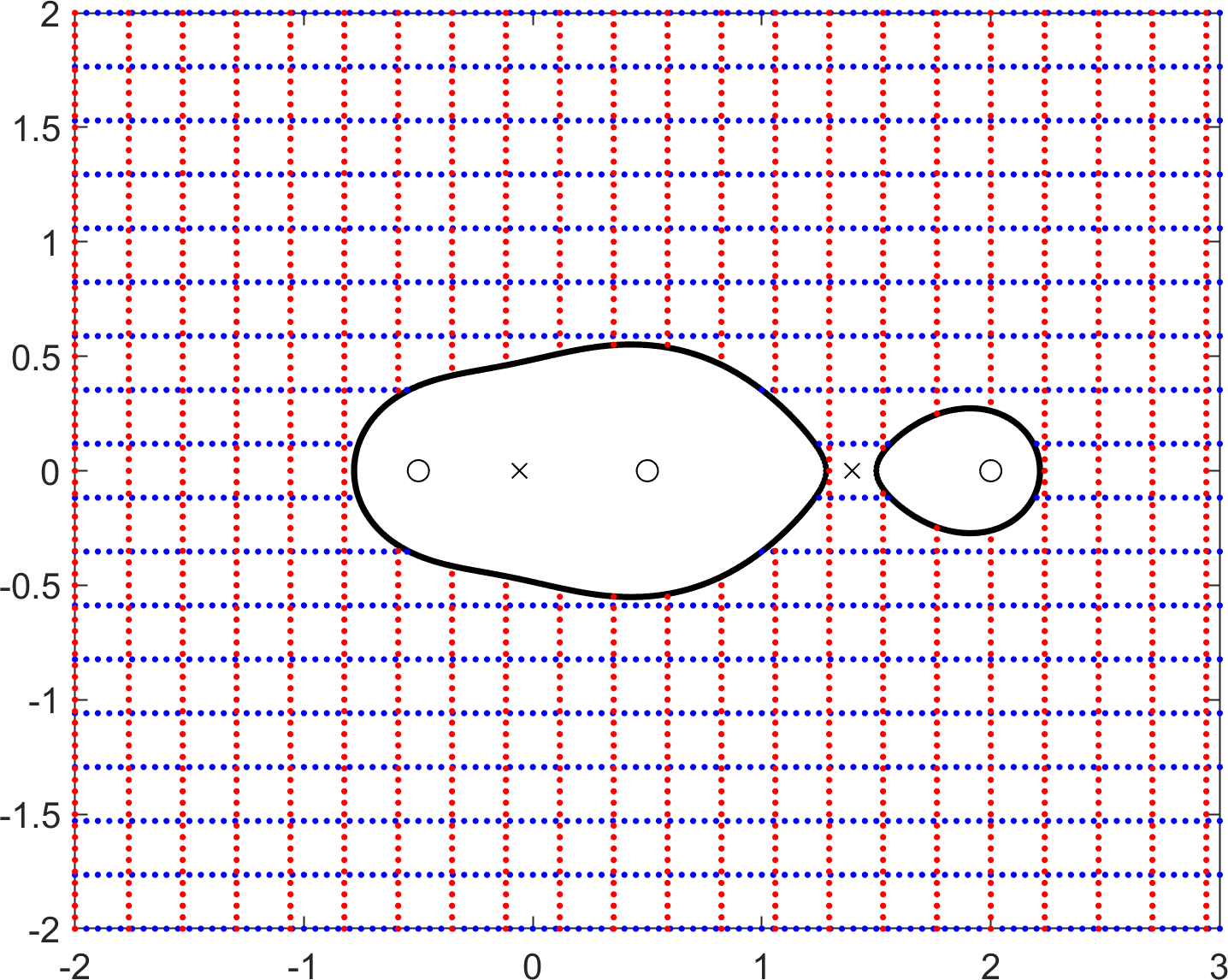}
\quad
\includegraphics[width=0.4\linewidth]{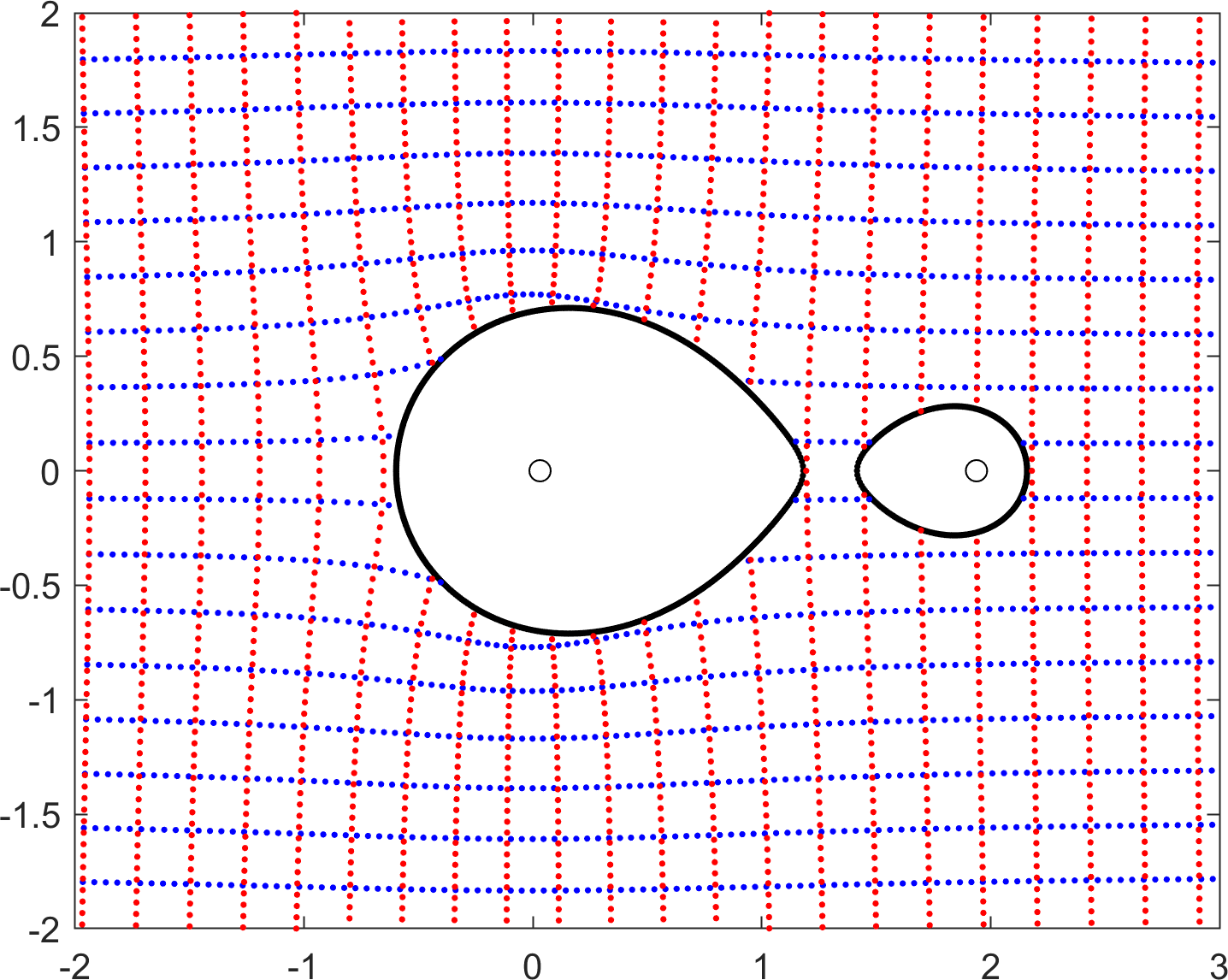}

}
\caption{Pre-image $E = P_3^{-1}(\overline{\bD})$ in Example~\ref{ex:disk}.
Left: $\partial E$ (black line), zeros of $P_3$ (circles) and $P_3'$ (crosses), 
and a cartesian grid.
Right: $\partial L$ (black line), $a_1, a_2$ (circles) and the image of the 
grid under $\Phi$.}
\label{fig:disk}
\end{figure}

Finally, we consider an example for Theorem~\ref{thm:aj_for_two_components}.

\begin{example} \label{ex:disk}
For $\alpha, \beta \in \C$, consider the polynomial
\begin{equation*}
P_3(z) = (z - \alpha) (z^2 - \beta^2)
= z^3 - \alpha z^2 - \beta^2 z + \alpha \beta^2
\end{equation*}
of degree $n = 3$.  The critical points of $P_3$ are
\begin{equation*}
z_\pm = \frac{\alpha \pm \sqrt{\alpha^2 + 3 \beta^2}}{3}.
\end{equation*}
In the case $\alpha = 2$ and $\beta = 1/2$, we have $P_3(z_-) \approx 0.5076 
\in \overline{\bD}$ and $P_3(z_+) \approx 1.9375 \in \C \setminus 
\overline{\bD}$, hence $E = 
P_3^{-1}(\overline{\bD})$ has $\ell = 2$ components by 
Theorem~\ref{thm:E_decomposition}; see Figure~\ref{fig:disk} (left). 
Note that $\comp{E}$ is not a lemniscatic domain (in contrast to the case 
considered in Example~\ref{ex:preim_of_unit_disk_with_n_components}).
Write $E = E_1 \cup E_2$, where $E_1$ is the component on the left (with 
$\pm \beta \in E_1$).  Then $m_1 = 2/3$ and $m_2 = 1/3$ by 
Theorem~\ref{thm:exponents_poly_preimage}.
Moreover, $E_1^* = E_1$ and $E_2^* = E_2$, since $P_n$ is real and 
$\overline{\bD}$ is symmetric with respect to the real line,
which implies that $a_1, a_2 \in \R$ by Theorem~\ref{thm:aj_real}.
Then, by Theorem~\ref{thm:aj_for_two_components},
\begin{equation*}
\left( a_2 - \frac{\alpha}{3} \right)^3 = - 2 P_3(z_+) \in \R.
\end{equation*}
Since $a_2 - \frac{\alpha}{3}$ is real, taking the real third root yields
\begin{equation*}
a_2 = \frac{\alpha}{3} + \sqrt[3]{-2 P_3(z_+)} \approx 1.9375
\quad \text{and} \quad
a_1 = \frac{1}{3} \alpha - \frac{1}{2} \sqrt[3]{-2 P_3(z_+)} \approx 0.0313.
\end{equation*}
Moreover, $\capacity(E) = 1$ by~\eqref{eqn:capacity_poly_preimage}, hence
\begin{equation*}
L = \{ w \in \C : \abs{w - a_1}^{2/3} \abs{w-a_2}^{1/3} \leq 1 \}.
\end{equation*}
Here, $Q(w) = (w-a_1)^2 (w-a_2)^1$, hence
\begin{equation*}
\Phi(z) = Q^{-1}(P_3(z)),
\end{equation*}
with a branch of $Q^{-1}$ such that $\Phi(z) = z + \cO(1/z)$ at infinity.
Here, we can obtain the boundary values of $\Phi$ for $z \in \partial E$ by 
solving $Q(w) = P_3(z)$ and identifying the boundary points in the correct way.
Then, since $\Phi(z) - z$ is analytic in $\comp{E}$ and zero 
at infinity, we have
\begin{equation} \label{eqn:Phi_as_intergal}
\Phi(z) = z + \frac{1}{2 \pi i} \int_{\partial E} \frac{\Phi(\zeta) - 
\zeta}{\zeta - z} \, d\zeta, \quad z \in \C \setminus E,
\end{equation}
where $\partial E$ is negatively oriented, such that $\comp{E}$ 
lies to the left of $\partial E$.  Figure~\ref{fig:disk} also shows a cartesian 
grid (left) and its image under $\Phi$ (right).  For the computation, we 
numerically approximate the integral in~\eqref{eqn:Phi_as_intergal} with the 
trapezoidal rule.
\end{example}

\appendix
\section{Appendix}

\begin{lemma} \label{lem:symmetry_of_Green_function}
Let $E \subseteq \C$ be compact, such that $\comp{E}$ has a 
Green's function $g_E$.
If $E^* = E$, then $g_E(\conj{z}) = g_E(z)$ and $\partial_z g_E(\conj{z}) = 
\conj{\partial_z g_E(z)}$.
Moreover, the critical points of $g_E$ are real or appear in complex conjugate 
pairs.
\end{lemma}

\begin{proof}
Since $E^* = E$, the function $z \mapsto g_E(\conj{z})$ is also a Green's 
function with pole at infinity of $\comp{E}$, hence 
$g_E(\conj{z}) = g_E(z)$ for all $z \in \comp{E}$ by the 
uniqueness of the Green's function.
% With the definition of the Wirtinger derivatives, this also implies 
% $\partial_z g_E(\conj{z}) = \conj{\partial_z g_E(z)}$.
Write $g(x,y) = g_E(z)$, then $g(x,y) = g(x,-y)$ and
\begin{equation} \label{eqn:partial_deriv_of_green_function}
\frac{\partial g}{\partial x}(x,y) = \frac{\partial g}{\partial x}(x,-y),
\quad
\frac{\partial g}{\partial y}(x,y) = - \frac{\partial g}{\partial y}(x,-y),
\end{equation}
hence
\begin{equation*}
2 \partial_z g_E(z)
= \frac{\partial g}{\partial x}(x,y) - i \frac{\partial g}{\partial y}(x,y)
= \frac{\partial g}{\partial x}(x,-y) + i \frac{\partial g}{\partial y}(x,-y)
= \conj{2 (\partial_z g_E)(\conj{z})}.
\end{equation*}
The critical points of $g_E$ are the zeros of the analytic function $\partial_z 
g_E$.
Since $\partial_z g_E(\conj{z}) = \conj{\partial_z g_E(z)}$, if $z_*$ is a zero 
of $\partial_z g_E$ then also $\conj{z}_*$ is a zero.
\end{proof}

\begin{lemma} \label{lem:compact_cap_R_is_connected}
Let $K \subseteq \C$ be a non-empty compact, simply connected set with $K^* = 
K$, then $K \cap \R$ is either an interval or a single point.
\end{lemma}

\begin{proof}
Since $K^* = K$ and $K$ is connected, $K \cap \R$ is not empty.
Since $\comp{K}$ is connected, $K \cap \R$ must be connected
(otherwise the symmetry and simply-connectedness of $K$ would imply that 
$\comp{K}$ is not connected).
Thus, $K \cap \R$ is a point or an interval.
\end{proof}

\begin{lemma} \label{lem:symmetric_integral_is_real}
Let $\gamma$ be a smooth Jordan curve symmetric with respect to the real line 
and let $f$ be integrable with $f(\conj{z}) = \conj{f(z)}$ on $\gamma$.  Then
\begin{equation*}
\frac{1}{2 \pi i} \int_\gamma f(z) \, dz \in \R.
\end{equation*}
\end{lemma}

\begin{proof}
Since $\gamma$ is symmetric with respect to the real line, we can write $\gamma 
= \gamma_1 + \gamma_2$ with $\gamma_2 \coloneq - \conj{\gamma}_1$.  Then
\begin{equation*}
\begin{split}
\int_\gamma f(z) \, dz
&= \int_{\gamma_1} f(z) \, dz - \int_{\conj{\gamma}_1} f(z) \, dz
= \int_a^b f(\gamma_1(t)) \cdot \gamma_1'(t) \, dt - 
\int_a^b f(\conj{\gamma_1(t)}) \cdot \conj{\gamma_1'(t)} \, dt \\
&= 2i \int_a^b \im (f(\gamma_1(t)) \cdot \gamma_1'(t)) \, dt,
\end{split}
\end{equation*}
which yields the result.
\end{proof}

Though the following theorem must be known, we did not find it in the 
literature.  For completeness, we include a proof.

\begin{theorem} \label{thm:cK_is_connected}
Let $P$ be a non-constant polynomial and $\Omega \subseteq \C$ be a 
simply connected compact set.  Then $\widehat{\C} \setminus P^{-1}(\Omega)$ is 
open and connected, i.e., a region.
\end{theorem}

\begin{proof}
Clearly, $G \coloneq \widehat{\C} \setminus P^{-1}(\Omega) = 
P^{-1}(\widehat{\C} \setminus \Omega)$ is open and contains $\infty$.
Let $G_\infty \subseteq G$ be that component of $G$ that contains $\infty$.
Suppose that $G$ is not connected, i.e., $G \neq G_\infty$.  Then there exists 
another component $G_1 \subseteq G$, and $G_1$ is a bounded region.
Then $P(G_1)$ is a bounded region with $P(G_1) \subseteq \widehat{\C} \setminus 
\Omega$.

Next, we show that $\partial P(G_1) \subseteq \partial \Omega$.
Let $w \in \partial P(G_1)$.  Then there exists $w_k \in P(G_1)$ with $w_k \to 
w$.  For each $k$, there exists $z_k \in G_1$ with $P(z_k) = w_k$.
Since $G_1$ is bounded, the sequence $(z_k)_k$ has a convergent subsequence 
$(z_{k_j})_j$ with $z_{k_j} \to z \in \overline{G}_1$.
This implies that $P(z) = w$.
Since $w \in \partial P(G_1)$, we have $z \in \partial G_1$ (otherwise, $z \in 
G_1$ would imply $P(z) \in P(G_1)$ and, since $P(G_1)$ is open, $w = P(z) 
\notin \partial P(G_1)$).
Since $G$ is open, this implies that $z \notin G$ and hence that $z \in 
P^{-1}(\Omega)$ and $w = P(z) \in \Omega$.
Since $w_k \in P(G_1) \subseteq \widehat{\C} \setminus \Omega$ and $w_k \to w$, 
we obtain that $w \in \partial \Omega$.

We have shown that $P(G_1) \subseteq \widehat{\C} \setminus \Omega$ is a region 
with $\partial P(G_1) \subseteq \partial \Omega = \partial (\widehat{\C} 
\setminus \Omega)$.
Since $\widehat{\C} \setminus \Omega$ is connected, this implies that $P(G_1) = 
\widehat{\C} \setminus \Omega$, which contradicts that $P(G_1)$ is bounded.
This shows that $G = \widehat{\C} \setminus P^{-1}(\Omega)$ is connected.
\end{proof}

\small
\bibliographystyle{siam}
\bibliography{walshmap.bib}

\end{document}